\documentclass[a4paper, 11pt, reqno]{amsart}
\usepackage{amsmath,amsfonts,amssymb,amsthm,enumerate}
\usepackage{graphicx}
\usepackage{xy} \xyoption{all}
\usepackage{xcolor}

\newtheorem{thm}{Theorem}[section]
\newtheorem{cor}[thm]{Corollary}
\newtheorem{prop}[thm]{Proposition}
\newtheorem{lem}[thm]{Lemma}

\newcommand{\varep}{\varepsilon}

\theoremstyle{definition}
\newtheorem{defn}[thm]{Definition}
\newtheorem{exas}[thm]{Example}
\newtheorem{rem}[thm]{Remark}

\let\phi\varphi

\begin{document}
\title{Realizing corners of Leavitt path algebras as Steinberg algebras, with corresponding connections to graph $C^*$-algebras}
\maketitle
\begin{center}
G.~Abrams\footnote{Department of Mathematics, University of Colorado, Colorado Springs,
Colorado, USA. E-mail address: \texttt{abrams@math.uccs.edu}}, M. Dokuchaev\footnote{Instituto de Matem\'{a}tica e Estat\'{i}stica, Universidade de S\~{a}o Paulo, Caixa
Postal 66281, S\~{a}o Paulo, SP 05315-970, Brazil. E-mail address: \texttt{dokucha@gmail.com}} and 
T.\,G.~Nam\footnote{Institute of Mathematics, VAST, 18 Hoang Quoc Viet, Cau Giay, Hanoi, Vietnam. E-mail address: \texttt{tgnam@math.ac.vn}

\ \ {\bf Acknowledgements}:   
The second author was partially supported by CNPq of Brazil 
Proc. 307873/2017-0 and by Fapesp of Brazil Proc. 2015/09162-9. The third author was supported by FAPESP of Brazil Proc. 2018/06538-6.} 
\end{center}

\begin{abstract} We show that the endomorphism ring of any nonzero finitely generated projective module over the  Leavitt path algebra $L_K(E)$ of an arbitrary graph $E$ with coefficients in a field $K$ is isomorphic to a Steinberg algebra. This yields in particular that every nonzero corner of the Leavitt path algebra of an arbitrary graph is isomorphic to a Steinberg algebra.  This in its turn gives that every $K$-algebra with local units  which is Morita equivalent to the Leavitt path algebra of  a row-countable graph is  isomorphic to a Steinberg algebra. Moreover, we prove that a corner  by a projection of a $C^*$-algebra of a countable graph is isomorphic to the $C^*$-algebra of an ample groupoid.
\medskip



\end{abstract}

\section{Introduction and Preliminaries}
Given a row-finite directed graph $E$ and any field $K$, the first author and Aranda Pino in \cite{ap:tlpaoag05}, and independently Ara, Moreno, and Pardo in \cite{amp:nktfga}, introduced the \emph{Leavitt path algebra} $L_K(E)$.  The first author and Aranda Pino later extended the definition in \cite{ap:tlpaoag08} to all countable directed graphs. Goodearl in \cite{g:lpaadl} extended the
notion of Leavitt path algebras $L_K(E)$ to all (possibly uncountable) directed graphs $E$. Leavitt path algebras generalize the Leavitt algebras $L_K(1, n)$ of \cite{leav:tmtoar}, and also contain many other interesting classes of algebras. In addition, Leavitt path algebras are intimately related to graph $C^*$-algebras (see \cite{r:ga}). During  the past fifteen years, Leavitt path algebras have become a topic of intense investigation by mathematicians from across the mathematical spectrum.  
For a detailed history and overview of  Leavitt path algebras we refer the reader to the survey article  \cite{a:lpatfd}.  

One of the interesting questions in the theory of Leavitt path algebras is to find relationships between graphs $E$ and $F$ such that their corresponding Leavitt path algebras are Morita equivalent.
This question has been investigated in numerous articles, see e.g. \cite{alps:fiitcolpa} and \cite{RT}.  
In large part this research effort has been motivated by the goal of resolving
the ``Morita equivalence conjecture": if the Leavitt path algebras of the graphs $E$ and $F$ are Morita equivalent, must the corresponding graph $C^*$-algebras for $E$ and $F$  also be 
(strongly) Morita equivalent? (see \cite[p. 3758]{at:iameoga}). 

If $e$ is a nonzero idempotent in the ring $A$, then the ``corner of $A$ generated by $e$ is the unital ring $eAe$.   Since Morita equivalence passes to (full) corners, it may be germane to the resolution of  the Morita equivalence conjecture to understand situations 
in which a corner of a Leavitt path algebra is again a Leavitt path algebra, or at least Morita equivalent to a Leavitt path algebra. In \cite{an:colpaofgalpa} the first and third authors proved that any corner of the Leavitt path algebra $L_K(E)$ of a finite graph $E$ is isomorphic to a Leavitt path algebra $L_K(F)$, and $F$ may be obtained from $E$ via a sequence of well-understood ``graph transformations".   
This result notwithstanding, it turns out that a corner of a Leavitt path algebra of an arbitrary graph $E$ need not in general be isomorphic to a Leavitt path algebra (see Example \ref{cornernotleav}).  However, as a consequence of our main result (Theorem \ref{End(Q)general}), we establish in Corollary \ref{cornercor} that every  corner of such Leavitt path algebras is in fact isomorphic to an algebra of a more general type, to wit, a Steinberg algebra.  

Steinberg algebras were introduced by Steinberg in \cite{stein:agatdisa}, and independently Clark et al.  in \cite{cfst:aggolpa}.   We begin by establishing that the corner algebra (or ``corner-like" algebra) of a Steinberg algebra generated by idempotents of a special form 
is again a Steinberg algebra (Proposition \ref{cornerStein}). Consequently, we obtain that corners of Leavitt path algebras of arbitrary graphs which arise as sums of finitely many distinct vertices are isomorphic to Steinberg algebras (Theorem \ref{gecorthem} and Corollary \ref{limofcor}). 

Let $Q$ denote a finitely generated projective left module over the Leavitt path algebra $L_K(E)$ for an arbitrary graph $E$ and field $K$.  Using  Theorem \ref{gecorthem}, we are subsequently able to establish our main result, that the endomorphism ring ${\rm End}_{L_K(E)}(Q)$ 
 is isomorphic to a Steinberg algebra (Theorem \ref{End(Q)general}). To do so, we use the theorem of Ara-Goodearl, and independently Hay et al.  (Theorem \ref{AG-HLMRTthm}) to have a description up to isomorphism of $Q$ as a finite direct sum of left $L_K(E)$-modules  of the forms $L_K(E)v$ and $L_K(E)(w-\sum_{e\in T}ee^*)$, where $v\in E^0$, $w\in E^0$ is an infinite emitter and $T$ is a non-empty finite subset of $s^{-1}(w)$. To analyze terms  of the form $L_K(E)(w-\sum_{e\in T}ee^*)$, we  extend to arbitrary graphs a result of the first author and his co-authors \cite[Theorem 2.8]{aalp:tcqflpa} about the out-split graph. By a sequence of out-splittings we get a graph $G$,  and obtain a description up to isomorphism of the $L_K(E)$-module $Q$ in terms of cyclic projective modules of the form $L_K(G)v$ where $v\in G^0$ (Proposition \ref{cateiso}). 
This result combined with 
 Theorem \ref{gecorthem} then establish Theorem \ref{End(Q)general}.   We note that the ample groupoid whose Steinberg algebra is isomorphic to the given corner of the Leavitt path algebra can be explicitly described.  
Theorem \ref{End(Q)general} has as an immediate consequence that any  corner of a Leavitt path algebra of an arbitrary graph is isomorphic to a Steinberg algebra (Corollary \ref{cornercor}).  As well, Theorem \ref{End(Q)general} and Corollary \ref{limofcor}  yield  Theorem \ref{moquiLeavitt}, showing that every $K$-algebra with local units that is Morita equivalent to the Leavitt path algebra of a row-countable graph  is   isomorphic to a Steinberg algebra. 

The development of Leavitt path algebras has both been a guide for, and been guided by, investigations into structures known as {\it graph} $C^*$-{\it algebras}.   As we remind the reader in Section 4, for a directed graph $E$, the Leavitt path algebra $L_{\mathbb{C}}(E)$ sits as a dense $\ast$-subalgebra in the graph $C^*$-algebra $C^*(E)$. As has been demonstrated by a number of remarkable results in both subjects over the past fifteen years, there is a very tight, although not completely well-understood, relationship between these two mathematical objects.  We find another example of this tight connection in the current work.   Specifically, by employing the same general approach as that used in the proof of  Theorem \ref{End(Q)general}, we establish in Theorem \ref{CorC-algtheo}  that a corner by a projection of a graph $C^*$-algebra of a countable graph is isomorphic to the $C^*$-algebra of an ample groupoid. We note that Arklint and Ruiz \cite{ar:cocka},  and Arklint, Gabe and Ruiz  \cite{ArklintGabeRuiz}  have established (among many other things) a specific case of this result.  To wit:  when $E$ is a countable graph having finitely many vertices, then any corner by a projection $p$ of $C^*(E)$ is in fact isomorphic  as $C^*$-algebras to a graph $C^*$-algebra;  that is, when $E$ is a countable graph having finitely many vertices, then $pC^*(E)p \cong C^*(F)$ for some graph $F$.

\medskip

We now present a streamlined version of the necessary background ideas.   We refer the reader to \cite{AF} for information about general ring-theoretic constructions, and to  \cite{AAS}  for additional information about Leavitt path algebras.

A (directed) graph $E = (E^0, E^1, s, r)$ 
consists of two disjoint sets $E^0$ and $E^1$, called \emph{vertices} and \emph{edges}
respectively, together with two maps $s, r: E^1 \longrightarrow E^0$.  The
vertices $s(e)$ and $r(e)$ are referred to as the \emph{source} and the \emph{range}
of the edge~$e$, respectively. A graph $E$ is called \emph{row-finite} if $|s^{-1}(v)|< \infty$ for all $v\in E^0$. 
A graph $E$ is called \emph{row-countable} if there are at most
countable arrows starting at any vertex. A graph $E$ is called \textit{countable} if both sets $E^0$ and $E^1$ are countable. 
A graph $E$ is \emph{finite} if both sets $E^0$ and $E^1$ are finite.  
A vertex~$v$ for which $s^{-1}(v)$ is empty is called a \emph{sink}; a vertex~$v$ is \emph{regular} if $0< |s^{-1}(v)| < \infty$; and a vertex~$v$ is an \textit{infinite emitter} if $|s^{-1}(v)| = \infty$.

A \emph{path of length} $n$ in a graph $E$ is a sequence $p = e_{1} \cdots e_{n}$  of edges $e_{1}, \dots, e_{n}$ such that $r(e_{i}) = s(e_{i+1})$ for $i = 1, \dots, n-1$.  In this case, we say that the path~$p$ starts at the vertex $s(p) := s(e_{1})$ and ends at the vertex $r(p) := r(e_{n})$, we write $|p| = n$ for the length of $p$.  We consider the elements of $E^0$ to be paths of length $0$. We denote by $E^{*}$ the set of all paths in $E$. An \textit{infinite path} in $E$ is an infinite sequence $p = e_1 \cdots e_n\cdots$ of edges in $E$ such that $r(e_i) = s(e_{i+1})$ for all $i\geq 1$. In this case, we say that the infinite path $p$ starts at the vertex $s(p) := s(e_1)$. We denote by $E^{\infty}$ the set of all infinite paths in $E$.

For an arbitrary graph $E = (E^0,E^1,s,r)$
and any  field $K$, the \emph{Leavitt path algebra} $L_{K}(E)$ {\it of the graph}~$E$
\emph{with coefficients in}~$K$ is the $K$-algebra generated
by the union of the set $E^0$ and two disjoint copies of $E^1$, say $E^1$ and $\{e^*\mid e\in E^1\}$, satisfying the following relations for all $v, w\in E^0$ and $e, f\in E^1$:
\begin{itemize}
\item[(1)] $v w = \delta_{v, w} w$;
\item[(2)] $s(e) e = e = e r(e)$ and
$r(e) e^* = e^* = e^*s(e)$;
\item[(3)] $e^* f = \delta_{e, f} r(e)$;
\item[(4)] $v= \sum_{e\in s^{-1}(v)}ee^*$ for any regular vertex $v$;
\end{itemize}
where $\delta$ is the Kronecker delta.

For any path $p= e_1e_2\cdots e_n$,  the element $e^*_n\cdots e^*_2e^*_1$ of $L_K(E)$ is denoted by $p^*$.   It can be shown (\cite[Lemma 1.7]{ap:tlpaoag05}) that $L_K(E)$ is  spanned as a $K$-vector space by $\{pq^* \mid p, q\in E^*, r(p) = r(q)\}$.  Indeed,  $L_K(E)$ is a $\mathbb{Z}$-graded $K$-algebra:  
$L_K(E)= \bigoplus_{n\in \mathbb{Z}}L_K(E)_n$,  where for each $n\in \mathbb{Z}$, the degree $n$ component $L_K(E)_n$ is the set \ $ \text{span}_K \{pq^*\mid p, q\in E^*, r(p) = r(q), |p|- |q| = n\}$.

\begin{rem}\label{oneedge&cor}  Let $K$ be any field, $E$ any graph and $H$ a finite subset of $E^0$.
	
(1) If $v \in E^0$ and $s^{-1}(v)$ is a single edge (say  $s^{-1}(v)= \{f\}$), then $ff^* = v$. 
	
(2)	$(\sum_{v\in H}v)L_K(E)(\sum_{v\in H}v)$ is  spanned as a $K$-vector space by $\{pq^*\mid p, q\in E^*, r(p) = r(q) \text{ and } s(p) , s(q) \in H\}$.
\end{rem}

A graph morphism from a graph $E$ to a graph $F$ is a pair $\phi = (\phi^0, \phi^1)$ consisting of maps $\phi^0: E^0\longrightarrow F^0$ and $\phi^1: E^1\longrightarrow F^1$ such that $s_F\phi^1 = \phi^0s_E$ and $r_F\phi^1 = \phi^0r_E$. A graph morphism $\phi: E\longrightarrow F$ is a \textit{CK-morphism} (short for \textit{Cuntz-Krieger morphism}) provided
\begin{itemize}
\item[(1)] $\phi^0$ and $\phi^1$ are injective;
\item[(2)] for any regular vertex $v\in E^0$, $\phi^1$ induces
a bijection $s^{-1}_E(v)\longrightarrow s^{-1}_F(\phi^0(v))$.
\end{itemize}
In particular, condition (1) says that $\phi$ maps $E$ isomorphically onto a subgraph of $F$, while condition (2) implies that $\phi^0$ must send regular vertices to regular vertices. If $E$ is row-finite, injectivity of $\phi^0$ together with condition (2) is sufficient to ensure injectivity of $\phi^1$. Thus, in this case, $\phi$ is a CK-morphism if and only if it is a \textit{complete graph homomorphism} in the sense of \cite[p. 161]{amp:nktfga}.

A subgraph $E$ of a graph $F$ is called a {\it CK-subgraph} in case the inclusion map $E\longrightarrow F$ is a CK-morphism. Less formally:  $E$ is a CK-subgraph in case for every regular vertex $v$ of $E$, all edges which $v$ emits in $F$ are included in $E$.

We shall let \textbf{CKGr} denote the category of directed graphs: the objects of \textbf{CKGr} are arbitrary directed graphs, and the morphisms are arbitrary CK-morphisms.

Fix a field $K$, and consider a CK-morphism $\phi: E\longrightarrow F$ between graphs $E$ and $F$. Following \cite[p. 173]{g:lpaadl}, $\phi$ induces a $K$-algebra homomorphism $L_K(\phi): L_K(E)\longrightarrow L_K(F)$ defined by $L_K(\phi)(v) = \phi^0(v)$ for all $v\in E^0$ and $L_K(\phi)(e) = \phi^1(e)$ for all $e\in E^1$. The assignments $E\longmapsto L_K(E)$ and $\phi \longmapsto L_K(\phi)$ define a functor $L_K$ from the category of graphs \textbf{CKGr} to the category of $K$-algebras $K$-\textbf{Alg}.


\begin{prop}\label{Lpaproperties}  
	
$(1)$ ({\cite[Lemma 2.5 (a)]{g:lpaadl}})  Arbitrary direct limits exist in the category \textbf{CKGr}. In particular, if $((E_i)_{i\in I}, (\phi_{ij})_{i\leq j\ {\rm in}\ I})$ is a directed system in \textbf{CKGr} and $\varinjlim_I E_i = E$, then 
$E^0$ and $E^1$ are the direct limits of the corresponding direct systems of sets $((E^0_i)_{i\in I}, (\phi^0_{ij})_{i\leq j\ {\rm in}\ I})$ and $((E^1_i)_{i\in I}, (\phi^1_{ij})_{i\leq j\ {\rm in}\ I})$.
	
$(2)$  ({\cite[Lemma 2.5 (b)]{g:lpaadl}}) For any field $K$, the functor $L_K: {\rm \textbf{CKGr}}\longrightarrow K{-\rm \textbf{Alg}}$	preserve direct limits.
	
$(3)$ ({\cite[Lemma 1.6.6]{AAS}}) Let $K$ be any field and 
$\phi: E\longrightarrow F$ a CK-morphism between graphs $E$ and $F$. Then the map $L_K(\phi): L_K(E)\longrightarrow L_K(F)$, defined by $L_K(\phi)(v) = \phi^0(v)$ for all $v\in E^0$ and $L_K(\phi)(e) = \phi^1(e)$ for all $e\in E^1$, is an injective $K$-algebra homomorphism. Consequently, if $E$ is a CK-subgraph of $F$, then $L_K(E)$ is a subalgebra of $L_K(F)$.   
\end{prop} 

A ring $R$ is said to have \textit{local units} if every finite subset of $R$ is contained in a subring of the form $eRe$ where $e=e^{2}\in R$. For example, the Leavitt path algebra $L_K(E)$ of an arbitrary graph $E$ with coefficients in $K$ is a $K$-algebra with local units (see, e.g. \cite[Lemma 1.6]{ap:tlpaoag05}). We call a left module $M$ over $R$ unitary if $RM=M,$ $i.e.,$  for each $m\in M$ there are $r_{1},\ldots$, $r_{n}\in R$ and $m_{1},\ldots,$ $m_{n}\in M$ such that $r_{1}m_{1}+\cdots+r_{n}m_{n}=m$. If $R$ is a ring with local units then this implies that for every finite subset
$M^{\prime}\subset M$ there is an idempotent $e\in R$ such that $em=m$ for all $m\in M^{\prime}$. By $R$-Mod we denote the category of unitary left $R$-modules together with the usual $R$-homomorphisms. Since $R$-Mod is a Grothendieck category, various standard homological notions are applicable. In particular, projective modules make sense in $R$-Mod for this {\color{red} is} a categorical notion. We refer the reader to \cite[Section 10]{ag:lpaosg} for an investigation of projective modules over non-unital rings.

For any ring $R$ with local units, $\mathcal{V}(R)$ denotes the set of isomorphism classes
(denoted by $[P]$) of finitely generated projective left $R$-modules.   
$\mathcal{V}(R)$ is an abelian monoid with operation
$$[P] + [Q] = [P\oplus Q]$$ for any isomorphism classes $[P]$ and $[Q]$.   On the other hand, for any directed graph
$E=(E^0, E^1, s, r)$ the \textit{graph monoid} $M_E$ 
is defined as follows. Denote by $T$ the free abelian monoid (written additively) with generators
\begin{center}
$E^0 \sqcup \{q_{_Z} \mid Z$ is a nonempty finite subset of $s^{-1}(v)$ and $v$ is an infinite emitter$\},$
\end{center}
and define relations on $T$ by setting
\begin{itemize}
\item[(1)] $v = \sum_{e\in s^{-1}(v)}r(e)$ for all regular vertex $v$;
\item[(2)] $v = \sum_{e\in Z}r(e) + q_{_Z}$ for all infinite emitters $v\in E^0$ and all nonempty finite subsets $Z\subset s^{-1}(v)$;
\item[(3)] $q_{_Z} = \sum_{e\in W\setminus Z}r(e) + q_{_W}$ for all nonempty finite sets $Z\subseteq W\subset s^{-1}(v)$, where $v\in E^0$ is an infinite emitter.
\end{itemize}
 Let $\sim_E$ be the congruence relation on $T$ generated
by these relations. Then $M_E$ is defined to be the quotient monoid $ T/_{\sim_E}$; we denote an element of $M_E$ by $[x]$,
where $x\in T$.  The foundational result about Leavitt path algebras for our work is the following:

\begin{thm}[{\cite[Theorem 4.3]{ag:lpaosg} and \cite[Theorem 4.9]{hlmrt:nsktflpa}}]\label{AG-HLMRTthm}
Let $E$ be an arbitrary graph and $K$ any field. Then there exists a monoid isomorphism $\gamma_E: M_E\longrightarrow \mathcal{V}(L_{K}(E))$ such that $\gamma_E([v])= [L_{K}(E)v]$ and 
$\gamma_E([q_Z]) = [L_{K}(E)(v - \sum_{e\in Z}ee^*)]$.
Specifically, the next three useful  consequences follow immediately.  

$(1)$ For any regular vertex $v\in E^0$, $L_K(E)v \cong   \bigoplus_{e\in s^{-1}(v)}L_K(E)r(e)$ as left $L_K(E)$-modules. 

$(2)$ For any infinite emitter $v\in E^0$ and finite sets $Z\subseteq W \subset s^{-1}(v)$, $$L_K(E)(v - \sum_{e\in Z}ee^*)\cong L_K(E)(v - \sum_{e\in W}ee^*) \bigoplus (\bigoplus_{f\in W\setminus Z}L_K(E)r(f))$$ as left $L_K(E)$-modules.

$(3)$ For any nonzero finitely generated projective left $L_K(E)$-module $Q$, there exist a nonempty finite subset $S\subseteq E^0\times 2^{E^1}$ and positive integers $\{n(v, T)\}_{(v, T)\in S}$ such that 
\begin{itemize}
\item[(i)] $Q\cong \bigoplus_{(v, T)\in S} n(v, T) L_K(E)(v- \sum_{e\in T}ee^*)$;
\item[(ii)] For all $(v, T)\in S$, $T$ is a finite subset of $s^{-1}(v)$; 
\item[(iii)] For all $(v, T)\in S$, $T$ is nonempty only if $v$ is an infinite emitter.	
\end{itemize} 
\end{thm}

We emphasize that the direct sums indicated in the above theorem  are external direct sums.  Also, throughout the paper, for a positive integer $n$ and a left $R$-module $M$, the direct sum of $n$ copies of $M$ is denoted $nM$.








We finish the introductory section by presenting some results about endomorphism rings of modules over   rings with local units
which will be of great importance in this analysis.  
These can be found in \cite[Chapters 1, 2]{AF}; the proofs given there for unital rings go through verbatim in the more general setting of rings with local units.    For a left $R$-module ${}_RM$, we write $R$-endomorphisms of $M$ on the right (i.e., the side opposite the scalars); so for $f,g \in {\rm End}_R(M)$,  $(m)fg$ means ``first $f$, then $g$".  

\begin{prop}\label{End}   Let $R$ be a ring with local units.

$(1)$  Let $e,f$  be idempotents in $R$.  Then ${\rm Hom}_R(Re,Rf) \cong eRf$ as abelian groups. 
 
$(2)$ Suppose $P \cong  \bigoplus_{i=1}^n P_i$ as left $R$-modules.  Then 
$${\rm End}_R(P) \cong \biggl( {\rm Hom}_R(P_i, P_j)\biggr),$$
the ring of $n\times n$ matrices for which the entry in the $i$-th row, $j$-th column is an element of ${\rm Hom}_R(P_i, P_j)$, for all $1\leq i,j \leq n$.    In particular, if $P\cong  \bigoplus_{i=1}^n Re_i$, an external direct sum of the left $R$-modules $Re_i$ for idempotents $e_i$, then
$${\rm End}_R(P) \cong \biggl(e_iRe_j\biggr),$$
the ring of $n\times n$ matrices for which the entry in the $i$-th row, $j$-th column is an element of $e_iRe_j$, for all $1\leq i,j \leq n$. 
\end{prop}

%
%
%
%

\medskip

\section{Corners of Steinberg algebras}
We begin this section by reminding the reader the basics of Steinberg algebras.   We then establish that for idempotents of a specified type, a corner algebra of a Steinberg algebra generated by such an idempotent is isomorphic to a Steinberg algebra
(Proposition \ref{cornerStein}).  
(This allows us to re-establish Webster's result that the path space of an arbitrary graph is locally compact Hausdorff (Theorem \ref{Websterthm}).)   Consequently, we use Proposition \ref{cornerStein} to establish that, for idempotents of a specific form, the corner of a Leavitt path algebra of an arbitrary graph generated by such an idempotent is isomorphic to a Steinberg algebra (Theorem \ref{gecorthem} and Corollary \ref{limofcor}).

We begin by recalling the concepts of ample groupoids and Steinberg algebras. A \textit{groupoid} is a small category in which every morphism is invertible. It can also be viewed as a generalization of a group which has a partial binary operation. 
Let $\mathcal{G}$ be a groupoid. If $x \in \mathcal{G}$, $s(x) = x^{-1}x$ is the source of $x$ and $r(x) = xx^{-1}$ is its range. The pair $(x,y)$ is is composable if and only if $r(y) = s(x)$. The set $\mathcal{G}^{(0)}:= s(\mathcal{G}) = r(\mathcal{G})$ is called the \textit{unit space} of $\mathcal{G}$. Elements of $\mathcal{G}^{(0)}$ are units in the sense that $xs(x) = x$ and
$r(x)x = x$ for all $x\in \mathcal{G}$. For $U, V \subseteq \mathcal{G}$, we define
\begin{center}
$UV = \{\alpha\beta \mid \alpha\in U, \beta\in V \text{\, and \,} r(\beta) = s(\alpha)\}$ and $U^{-1} = \{\alpha^{-1}\mid \alpha\in U\}$.
\end{center}

A \textit{topological groupoid} is a groupoid endowed with a topology under which the inverse map is continuous, and such that the composition is continuous with respect to the relative topology on $\mathcal{G}^{(2)} := \{(\beta, \gamma)\in \mathcal{G}^2\mid s(\beta) = r(\gamma)\}$ inherited from $\mathcal{G}^2$. An \textit{\'{e}tale groupoid} is a topological groupoid $\mathcal{G},$ whose unit space  $\mathcal{G}^{(0)}$ is locally compact Hausdorff, and such that 
 the domain map $s$ is a local homeomorphism. In this case, the range map $r$ and the multiplication map are local homeomorphisms and  $\mathcal{G}^{(0)}$ is open in $\mathcal{G}$ \cite{r:egatq}. 

An \textit{open bisection} of $\mathcal{G}$ is an open subset $U\subseteq \mathcal{G}$ such that $s|_U$ and $r|_U$ are homeomorphisms onto an open subset of $\mathcal{G}^{(0)}$.  Similar to \cite[Proposition 2.2.4]{p:gisatoa} we have that $UV$ and $U^{-1}$ are compact open bisections for all compact open bisections $U$ and $V$ of an  \'{e}tale groupoid  $\mathcal{G}$.
An \'{e}tale groupoid $\mathcal{G}$ is called \textit{ample} if  $\mathcal{G}$ has a base of compact open bisections for its topology.

Steinberg algebras were introduced in \cite{stein:agatdisa} in the context of discrete inverse semigroup algebras and independently in \cite{cfst:aggolpa} as a model for Leavitt path algebras. Let $\mathcal{G}$ be an ample groupoid, and $K$ a field with the discrete topology. We denote by $K^{\mathcal{G}}$ the set of all continuous functions from $\mathcal{G}$ to $K$. Canonically,  $K^{\mathcal{G}}$ has the structure of a $K$-vector space with
operations defined pointwise.

\begin{defn}[{The Steinberg algebra}] 
Let $\mathcal{G}$ be an ample groupoid, and $K$ any field.   Let $A_K(\mathcal{G})$ be the $K$-vector subspace of $K^{\mathcal{G}}$ generated by the set
\begin{center}
$\{1_U\mid U$ is a compact open bisection of $\mathcal{G}\}$,
\end{center}
where $1_U: \mathcal{G}\longrightarrow K$ denotes the characteristic function on $U$. The \textit{multiplication} of $f, g\in A_K(\mathcal{G})$ is given by the convolution
\[(f\ast g)(\gamma)= \sum_{\gamma = \alpha\beta}f(\alpha)g(\beta)\] for all $\gamma\in \mathcal{G}$. The $K$-vector subspace $A_K(\mathcal{G})$, with convolution, is called the \textit{Steinberg algebra} of $\mathcal{G}$ over $K$.
\end{defn}

It is useful to note that \[1_U\ast 1_V = 1_{UV}\] for compact open bisections $U$ and $V$ (see \cite[Proposition 4.5]{stein:agatdisa}).

If $F$ is an orthogonal set of nonzero idempotents in a ring $R$, then the set $\sum_{f,f' \in F}fRf'$  is a  (not-necessarily-unital) subring of $R$, which we call a ``corner-like" subring.   In case $F$ is finite, then $e = \sum_{f\in F}f$ is an idempotent in $R$, and  $\Sigma_{f,f' \in F}fRf' = eRe$, the usual corner of $R$  generated by $e$.  

\begin{defn}\label{GsubUdef}
Let $\mathcal{G}$ be an ample groupoid, and let $U$ be an open subset of $\mathcal{G}^{(0)}$.  We define   $\mathcal{G}_U \subseteq \mathcal{G}$ by setting 
$$\mathcal{G}_U:= r^{-1}(U) \cap s^{-1}(U).$$
\end{defn}

\begin{prop}\label{cornerStein}
Let $K$ be a field and $\mathcal{G}$ an ample groupoid.  

$(1)$ If $U$ is an open subset of $\mathcal{G}^{(0)}$, then $\mathcal{G}_U$
 is an open ample subgroupoid of $\mathcal{G}$.

$(2)$ If $\{U_i\mid i\in I\}$ is a set of compact open subsets of $\mathcal{G}^{(0)}$ with $U_i\cap U_j = \emptyset$ for all $i\neq j$ and $U = \bigcup_{i\in I}U_i$, then $A_K(\mathcal{G}_U)$ can be realized as a corner-like subring of $A_K(\mathcal{G})$, specifically, $$A_K(\mathcal{G}_U) = \sum_{i, j\in I}1_{U_i}A_K(\mathcal{G})1_{U_j} .$$ 

\end{prop}
\begin{proof}
(1) We have that $\mathcal{G}_U$ is clearly a subset of $\mathcal{G}$ closed under the inversion and the composition, so $\mathcal{G}_U$ is a subgroupoid of $\mathcal{G}$ having $\mathcal{G}_U^{(0)} = U$. Since $U$ is open in $\mathcal{G}^{(0)}$, and the two maps $r$ and $s: \mathcal{G} \longrightarrow \mathcal{G}^{(0)}$ are continuous, $\mathcal{G}_U$ is open in $\mathcal{G}$. This implies that 
$\mathcal{G}_U$ is a topological subgroupoid of $\mathcal{G}$ with the subspace  topology.

We note that if $B$ is a compact open bisection of $\mathcal{G}_U$ then $B$ is also a compact open bisection of $\mathcal{G}$. Indeed, since $B$ is open in $\mathcal{G}_U$ and $\mathcal{G}_U$ is open in $\mathcal{G}$, $B$ is open in $\mathcal{G}$. Assume that $\{C_j\}_{j\in J}$ is an open cover of $B$ in $\mathcal{G}$. We then have that $\{C_j\cap \mathcal{G}_U\}_{j\in J}$ is an open cover of $B$ in $\mathcal{G}_U$. Since $B$ is a compact subset of $\mathcal{G}_U$, $B$ has some finite subcover $\{C_{j_k}\cap \mathcal{G}_U\}_{k=1}^n$, so $\{C_{j_k}\}_{k=1}^n$ is a finite subcover of $B$ in $\mathcal{G}$, and hence $B$ is a compact subset of $\mathcal{G}$. Since $B$ is an open bisection of $\mathcal{G}_U$, $s|_B$ and $r|_B$ are homeomorphims onto open subsets of $\mathcal{G}_U^{(0)}$, so $s|_B$ and $r|_B$ are homeomorphims onto open subsets of $\mathcal{G}^{(0)}$, since  $\mathcal{G}_U^{(0)}=U$ is open in $\mathcal{G}^{(0)},$  thus showing that $B$ is a compact open bisection of $\mathcal{G}$.

Let $\mathcal{B}$ be a base of compact open bisections for the topology on $\mathcal{G}$. We then have that $\mathcal{B}_U:=\{B\in \mathcal{B}\mid B\subseteq \mathcal{G}_U\}$ is a base of compact open bisections for the topology on $\mathcal{G}_U$. Indeed, every element $B\in \mathcal{B}_U$ is clearly a compact open bisection of $\mathcal{G}_U$. Let $V$ be any open subset of $\mathcal{G}_U$. Since $\mathcal{G}_U$ is open in $\mathcal{G}$, $V$ is open in $\mathcal{G}$. Then, since $\mathcal{B}$ is a base of compact open bisections for the topology on $\mathcal{G}$, there exists a subset $\{B_j\mid j\in J\} \subseteq \mathcal{B}$ such that $V = \cup_{j\in J}B_j$.
This implies that $B_j\subseteq \mathcal{G}_U$ for all $j\in J$, that means, $B_j\in \mathcal{B}_U$ for all $j\in J$, so $\mathcal{B}_U$ is a base of compact open bisections for the topology on $\mathcal{G}_U$. Therefore, $\mathcal{G}_U$ is an open ample subgroupoid of $\mathcal{G}$.

(2) We  have that 
\begin{center}
	$A_K(\mathcal{G})= \text{Span}_K\{1_B\mid B$ is a compact open bisection of $\mathcal{G}\}$
\end{center}
and
\begin{center}
	$A_K(\mathcal{G}_U)= \text{Span}_K\{1_{B'}\mid B'$ is a compact open bisection of $\mathcal{G}_U\}$. 
\end{center} 
Then, for any $i, j\in I$ we have
\begin{center}
$1_{U_i}A_K(\mathcal{G}) 1_{U_j}= \text{Span}_K\{1_{U_iBU_j}\mid B$ is a compact open bisection of $\mathcal{G}\}$
\end{center}
and each $U_iBU_j$ is a compact open bisection of $\mathcal{G}$ which is contained in $\mathcal{G}_U$, so $1_{U_i}A_K(\mathcal{G}) 1_{U_j} \subseteq A_K(\mathcal{G}_U)$, and hence $\sum_{i, j\in I} 1_{U_i}A_K(\mathcal{G}) 1_{U_j}\subseteq A_K(\mathcal{G}_U)$.

Let $V$ be a compact open bisection of $\mathcal{G}_U$. By the above note, $V$ is a compact open bisection of $\mathcal{G}$. Put $W = V^{-1}V\cup VV^{-1}$. We then have that $W$ is a compact open subset of $U$, so $W \subseteq \bigsqcup_{k=1}^n U_{i_k} =: B$ for some $i_1, i_2, \hdots, i_n\in I$. We then have $BVB = V$ and $1_V =1_B1_V1_B = (\sum_{k=1}^n1_{U_{i_k}}) 1_V(\sum_{k=1}^n1_{U_{i_k}}) = \sum_{1\leq k, j\leq n}1_{U_{i_k}}1_V1_{U_{i_j}}\in \sum_{i, j\in I} 1_{U_i}A_K(\mathcal{G}) 1_{U_j}$, so $A_K(\mathcal{G}_U)\subseteq \sum_{i, j\in I} 1_{U_i}A_K(\mathcal{G}) 1_{U_j}$, thus proving $A_K(\mathcal{G}_U) =\sum_{i, j\in I} 1_{U_i}A_K(\mathcal{G}) 1_{U_j}$, finishing the proof.
\end{proof}

\begin{defn}\label{Calphadef}
Let $E= (E^0, E^1, r, s)$ be a graph. For $\alpha \in E^*$ and a finite subset $G\subseteq s^{-1}(r(\alpha))$, we define
\begin{center}
$C(\alpha) = \{\alpha x \mid x\in E^*\cup E^{\infty},\, r(\alpha) = s(x)\}$ \ \ 
and  \ \ $C(\alpha, G) = C(\alpha)\setminus \bigcup_{e\in G}C(\alpha e).$
\end{center}
\end{defn}

\begin{lem}
Let $E$ be a graph and $\alpha, \beta\in E^{*}$ with $|\alpha|\geq |\beta|$. Then $C(\alpha) \cap C(\beta) \neq \emptyset$	if and only if $\alpha = \beta\alpha'$ for some $\alpha'\in E^{*}$.
\end{lem}
\begin{proof}
Assume that $C(\alpha) \cap C(\beta) \neq \emptyset$. Then, there exists an element $x\in E^*\cup E^{\infty}$ such that $x = \alpha x'$ and $x = \beta x''$ for some $x', x''\in E^*\cup E^{\infty}$. This implies that $\alpha = \beta\alpha'$ for some $\alpha'\in E^{*}$. The converse is obvious, finishing the proof.
\end{proof}

We define the map $\varphi: E^*\cup E^{\infty} \longrightarrow \{0, 1\}^{E^{*}}$ by $\varphi(p)(\alpha) = 1$ if $p\in C(\alpha)$, and $0$ otherwise. It is not hard to see that $\varphi$ is injective. Assume that $\{0, 1\}$ has the discrete topology. We endow $\{0, 1\}^{E^{*}}$ with the topology of pointwise convergence, and $E^*\cup E^{\infty}$ with the initial topology induced by $\{\varphi\}$. Then, the topological space $\{0, 1\}^{E^{*}}$ is compact by Tychonoff's Theorem, and Hausdorff because products preserve the Hausdorff property. Moreover, since $\varphi$ is injective, $\varphi$ is a homeomorphism onto its range.

In \cite[Theorem 2.1]{w:tpsoadg} Webster proved that for any countable graph $E$, $E^*\cup E^{\infty}$  is a locally compact Hausdorff space with the basis of compact open sets
$C(\alpha, G)$, where $\alpha\in E^*$ and $G$ is a finite subset of $s^{-1}(r(\alpha))$. In the following theorem we reproduce a short proof of Webster's result, generalized to the case where  $E$ is an arbitrary graph.

\begin{thm}[{cf. \cite[Theorem 2.1]{w:tpsoadg}}] \label{Websterthm}
For an arbitrary graph $E$, the collection
\begin{center}
$\{C(\alpha, G)\mid \alpha \in E^{*}, \, G \subseteq s^{-1}(r(\alpha))$ is finite$\}$
\end{center}	
is a basis for the locally compact Hausdorff topology on $E^*\cup E^{\infty}$.
\end{thm}
\begin{proof} The proof is essentially the same as that given in \cite[Theorem 2.1]{w:tpsoadg}, except when showing
the compactness of the sets $C(\alpha, G)$. First we consider the topology on $\{0, 1\}^{E^{*}}$. Given $\alpha \in E^*$ and disjoint finite subsets $F, G \subseteq E^{*}$, define 	$U^{F, G}_{\alpha} \subseteq \{0,1\}$ by setting 
\begin{equation*}
U^{F, G}_{\alpha}=  \left\{
\begin{array}{lcl}
\{1\}&  & \text{if } \alpha \in F\\
\{0\}&  & \text{if } \alpha \in G\\
\{0, 1\}&  & \text{otherwise .}
\end{array}%
\right.
\end{equation*}%
Then the sets $N(F, G) = \prod_{\alpha \in E^{*}}U^{F, G}_{\alpha}$, where $F$ and $G$ range over all finite, disjoint
pairs of subsets of $E^{*}$, form a basis for the topology on
$\{0, 1\}^{E^{*}}$. Therefore, the sets $\varphi^{-1}(N(F, G))$ form a basis for a topology on $E^* \cup E^\infty$. It is also not hard to check that \[\varphi^{-1}(N(F, G)) = (\bigcap_{\alpha \in F}C(\alpha))\setminus (\bigcup_{\beta\in G}C(\beta)).\]
If $\varphi^{-1}(N(F, G))$ is nonempty, then $\bigcap_{\alpha \in F}C(\alpha) \neq \emptyset$. By Lemma 2.5, there exists an element $\alpha \in F$ such that for any $\mu\in F$, $\alpha = \mu\mu'$ for some $\mu'\in E^*$. This implies that $\varphi^{-1}(N(F, G)) = C(\alpha)\setminus (\bigcup_{\beta\in G}C(\beta))$. Take any $\beta \in G$. If $\beta \neq \alpha p$ for all $p \in E^*,$ then we have  $C(\alpha) \cap C(\beta)$ is the empty set, or $C(\alpha)$ is contained in $C(\beta)$ (this only happens when $\alpha = \beta p$ for some $p \in E^*$). If the later happens, then $\phi^{-1}(N(F, G))$ is the empty set, which contradicts with our hypothesis that  $\phi^{-1}(N(F, G))$ is nonempty. So, we must have that  $C(\alpha) \cap C(\beta)$ is the empty set. This implies that  $\phi^{-1}(N(F, G)) = C(\alpha)\setminus (\cup_{\beta \in G' } C(\beta)),$ where $G' = \{\beta \in G \mid \beta = \alpha \beta'$ for some $\beta' \in E^*\}.$ We denote by $G''$ the set of all elements $\beta'$. Then  $\phi^{-1}(N(F, G)) = C(\alpha)\setminus (\cup_{p \in G''} C(\alpha p)) = C(\alpha, G'').$ Note that $G''$ is a subset of $r(\alpha)E^*$. From this note, without loss of generality we may assume that $G \subseteq r(\alpha)E^* := \{p\in E^*\mid s(p) = r(\alpha)\}$.

We claim that $\{C(\alpha, G)\mid \alpha \in E^{*}, \, G \subseteq s^{-1}(r(\alpha))$ is finite$\}$ and $\{C(\alpha, G)\mid \alpha \in E^{*}, \, G \subseteq r(\alpha)E^*$ is finite$\}$ are bases for the same topology. Indeed, the first set is clearly contained in the second one. Let $\alpha\in E^*$, a finite subset $G \subseteq r(\alpha)E^*$ and $\beta\in C(\alpha, G)$. Consider the following two cases.

\textit{Case} 1. $\beta= e_1 \cdots e_k\cdots$ is an infinite path. Let $n = \max\{|\alpha p|\mid p\in G\}$, $\mu = e_1\cdots e_n$ and $F = \emptyset$. We then have $\beta \in C(\mu) = C(\mu, F) \subseteq C(\alpha, G)$.

\textit{Case} 2. $\beta \in E^*$. If $G = \emptyset$ or both $\beta \neq \alpha$ and $\beta$ is not a prefix path of any path $\alpha p$ with $p\in G$, then we have $\beta \in C(\beta) = C(\beta, \emptyset)\subseteq C(\alpha, G).$ If $\beta \neq \alpha$ and $\beta$ is a prefix path of a path $\alpha p$ for some $p\in G$, then we assume that $p_1, p_2, \hdots, p_n$ are all such paths in $G$, that is, for each $1\leq i\leq n$, we have  $\alpha p_i = \beta q_i$ for some path $q_i$ of positive length. Write $q_i = f^i_1\cdots f^i_{k_i}$ ($1\leq i\leq n$). Let $F:= \{f^1_1, f^2_1, \hdots, f^n_1\} \subseteq s^{-1}(r(\beta))$. We then have $\beta \in C(\beta, F)\subseteq C(\alpha, G)$.
Consider the case when $\beta = \alpha$ and $G\neq \emptyset$. Assume that $p_1, p_2, \hdots, p_m$ are all elements of positive length in $G$. Write $p_i = e^i_1\cdots e^i_{n_i}$ ($1\leq i\leq m$). Set $F := \{e^1_1, e^2_1, \hdots, e^m_1\} \subseteq s^{-1}(r(\beta))$. We then have $\beta \in C(\beta, F) = C(\alpha, F) \subseteq C(\alpha, G)$.

In any case we always have that there exist a path $\mu\in E^*$ and a finite subset $F\subseteq s^{-1}(r(\mu))$ such that $\beta \in C(\mu, F)\subseteq C(\alpha, G)$, thus showing the claim.  Hence, the collection $\{C(\alpha, G)\mid \alpha \in E^{*}, \, G \subseteq s^{-1}(r(\alpha))$ is finite$\}$ is a basis for the
topology on $E^*\cup E^{\infty}$.

We prove that $E^*\cup E^{\infty}$ is a locally compact Hausdorff space. To do so, we will show that $C(\alpha, G)$ is compact for any $\alpha\in E^*$ and any finite subset $G\subseteq s^{-1}(r(\alpha))$. We first prove that $C(v)$ is compact for each $v\in E^0$. Since $\varphi$ is a homeomorphism
onto its range, and $\{0, 1\}^{E^*}$ is compact, it suffices to show that $\varphi(C(v))$ is closed. Let $\{p_d\in C(v)\mid d\in D\}$ be a net such that $\lim\varphi(p_d) = f \in \{0, 1\}^{E^*}$, where $D$ is a directed set. Write 
$A: =\{\alpha \in E^*\mid f(\alpha) = 1\}$. Since $v \in A$, we have that $A \neq \emptyset$. If $\alpha, \beta\in A$, then there exist $d_{\alpha}$, $d_{\beta}\in D$ such that $\varphi(p_d)(\alpha) = 1$ for all $d\in D$ with $d\geq d_{\alpha}$ and $\varphi(p_{d'})(\beta) = 1$ for all $d'\in D$ with $d'\geq d_{\beta}$. Since $D$ is a directed set, there must exist $d''\in D$ with $d''\geq d_{\alpha}$ and $d''\geq d_{\beta}$, so $\varphi(p_{d''})(\alpha) = 1 = \varphi(p_{d''})(\beta)$ and $p_{d''}\in C(\alpha) \cap C(\beta)$. By Lemma 2.5, we have either $\alpha = \beta\beta'$ for some $\beta'\in E^*$ or $\beta = \alpha\alpha'$ for some $\alpha'\in E^*$. This implies that all elements of $A$ determine a unique element $p\in E^*\cup E^{\infty}$.

We claim that $\lim \varphi(p_d) = \varphi(p)$. Let $\mu\in E^*$. If $\varphi(p)(\mu) =1$, then $p\in C(\mu)$, i.e, $p = \mu p'$ for some $p'\in E^*\cup E^{\infty}$. Then, there is an element $\gamma\in A$ such that $\gamma = \mu\gamma'$ for some $\gamma'\in E^*\cup E^{\infty}$. Since $\gamma\in A$, there exists $d_{\gamma}\in D$ such that $\varphi(p_d) (\gamma) = 1$, i.e, $p_d\in C(\gamma)$, for all $d\in D$ with $d\geq d_{\gamma}$, so $p_d\in C(\mu)$ for all $d\in D$ with $d\geq d_{\gamma}$. This implies that $\lim \varphi(p_d)(\mu) = 1=\varphi(p)(\mu)$. If $\varphi(p)(\mu) = 0$, then $p\notin C(\mu)$, i.e. $p \neq \mu x$ for all $x\in E^*\cup E^{\infty}$, so $\mu \notin A$, and hence $\lim \varphi(p_d)(\mu) = 0 = \varphi(p)(\mu)$, thus showing the claim. Since $\{0, 1\}^{E^*}$ is a Hausdorff space, $\lim \varphi(p_d)$ is determined uniquely,  so $f =\varphi(p)$. Therefore, $\varphi(C(v))$ is a closed set, so $C(v)$ is also closed.

We show that $C(\alpha)$ is compact for all $\alpha\in E^*$, by  induction on $|\alpha|$. If $\alpha = e\in E^1$, then $C(s(e))\setminus C(e) = C(s(e), \{e\})$ is an open set, so $C(e)$ is closed in $C(s(e))$, and hence it is compact. Assume that $\alpha = e_1\hdots e_n$ ($n\geq 2$). We then have $C(e_1\hdots e_{n-1})\setminus C(\alpha) = C(e_1\hdots e_{n-1}, \{e_n\})$ is an open set, so $C(\alpha)$ is closed in $C(e_1\hdots e_{n-1})$. By the induction hypothesis, $C(e_1\hdots e_{n-1})$ is compact, so   $C(\alpha)$ is compact.

Finally, for $\alpha\in E^*$ and a finite subset $G\subseteq s^{-1}(r(\alpha))$, we have $C(\alpha)\setminus C(\alpha, F) = \bigcup_{e\in F}C(\alpha e)$ is an open set, so $C(\alpha, F)$ is closed in $C(\alpha)$, and hence it is compact, thus finishing the proof.
\end{proof}

\begin{defn}\label{XsubEdefn}
Let $E= (E^0, E^1, r, s)$ be a graph. We define the subset $X_E$ of $ E^*\cup E^{\infty}$ by setting
\begin{center}
$X_E: = \{p \in E^{*}\mid r(p)$ is either a sink or an infinite emitter$\} \cup E^{\infty} $. 
\end{center}
\end{defn}

\noindent
We note that  if $\alpha\in (E^*\cup E^{\infty})\setminus X_E$, then $r(\alpha)$ is a regular vertex, and $C(\alpha, s^{-1}(r(\alpha))) =\{\alpha\}$
is open in $E^*\cup E^{\infty}$, thus showing $X_E$ is closed in $E^*\cup E^{\infty}$. From this and Theorem \ref{Websterthm}, we immediately get the following.  

\begin{cor} For any graph $E$, $X_E$ is  a locally compact Hausdorff space with the basis of compact open sets
\begin{center}
	$Z(\alpha, F) :=C(\alpha, F)\cap X_E$
\end{center}	
where $\alpha\in E^*$ and $F$ is a finite subset of $s^{-1}(r(\alpha))$.	
\end{cor}

We now describe the connection between Leavitt path algebras and Steinberg algebras.   
Let $E= (E^0, E^1, r, s)$ be a graph. We define  \[\mathcal{G}_E := \{(\alpha x, |\alpha| - |\beta|, \beta x)\mid \alpha, \beta \in E^*, x\in X_E, r(\alpha) = s(x) = r(\beta)\}.\] We view each $(x, k, y)\in \mathcal{G}_E$ as a morphism with range $x$ and source $y$. The formulas $(x, k, y) (y, l, z) = (x, k + l, z)$ and $(x, k, y)^{-1} = (y, -k, x)$ define composition and inverse maps on $\mathcal{G}_E$ making it a groupoid with $\mathcal{G}_E^{(0)} = \{(x, 0, x)\mid x\in X_E\}$ which we identify with the set $X_E$ by the map $(x, 0, x)\longmapsto x$. Then, by Corollary 2.8, we have that $\mathcal{G}_E^{(0)}$ is a locally compact Hausdorff space with the basis of compact open sets
\begin{center}
$Z(\alpha, \alpha, F) = \{(y, 0, y)\mid y\in Z(\alpha, F)\}$
\end{center}
where $\alpha\in E^*$ and $F$ is a finite subset of $s^{-1}(r(\alpha))$.

For $\alpha, \beta\in E^*$ with $r(\alpha) = r(\beta)$, and a finite subset $F\subseteq s^{-1}(r(\alpha))$, we define
\[Z(\alpha, \beta) = \{(\alpha x, |\alpha| - |\beta|, \beta x)\mid x\in X_E, r(\alpha) = s(x) = r(\beta)\}\] and \[Z(\alpha, \beta, F) = Z(\alpha, \beta)\setminus \bigcup_{e\in F}Z(\alpha e, \beta e).\]
By generalizing \cite[Lemma 2.5 and Proposition 2.6]{kprr:ggacka} (refer to \cite[Lemmas 2.15 and 2.16]{r:tgatlpa}) we get that the sets $Z(\alpha, \beta, F)$ form a basis  for a Hausdorff topology on $\mathcal{G}_E$. Let $r_{\mathcal{G}_E}$ and $s_{\mathcal{G}_E}: \mathcal{G}_E\longrightarrow \mathcal{G}_E^{(0)}$ be the range and source maps defined respectively by: $r_{\mathcal{G}_E}(x, k, y) = (x, 0, x)$ and $s_{\mathcal{G}_E}(x, 0, y) = (y, 0, y)$ for all $(x, k, y)\in \mathcal{G}_E$. For $\alpha\in E^*$ and a finite subset $F\subseteq s^{-1}(r(\alpha))$, we have
\begin{equation*}
\begin{array}{rcl}
r^{-1}_{\mathcal{G}_E}(Z(\alpha, \alpha, F)) &=& \bigcup_{\beta\in E^*,\,\, r(\beta) = r(\alpha)} Z(\alpha, \beta, F)
\end{array}
\end{equation*}
and 
\begin{equation*}
\begin{array}{rcl}
s^{-1}_{\mathcal{G}_E}(Z(\alpha, \alpha, F)) &=& \bigcup_{\beta\in E^*,\,\, r(\beta) = r(\alpha)} Z(\beta, \alpha, F),
\end{array}
\end{equation*}
so $r^{-1}_{\mathcal{G}_E}(Z(\alpha, \alpha, F))$ and $s^{-1}_{\mathcal{G}_E}(Z(\alpha, \alpha, F))$ are open in $\mathcal{G}_E$, and hence $r_{\mathcal{G}_E}, s_{\mathcal{G}_E}$ are continuous.
Then  for $\alpha, \beta\in E^*$ with $r(\alpha) = r(\beta)$, and a finite subset $F\subseteq s^{-1}(r(\alpha))$, we have that
$r_{\mathcal{G}_E}|_{Z(\alpha, \beta, F)}: Z(\alpha, \beta, F)\longrightarrow Z(\alpha, \alpha, F)$ and $s_{\mathcal{G}_E}|_{Z(\alpha, \beta, F)}: Z(\alpha, \beta, F)\longrightarrow Z(\beta, \beta, F)$ are homeomorphisms. Then, since $Z(\alpha,\alpha, F)$ is compact, $Z(\alpha, \beta, F)$ is compact. Thus, the sets $Z(\alpha, \beta, F)$ constitute a basis of compact open bisections for a topology under which
$\mathcal{G}_E$ is a Hausdorff ample groupoid (refer to \cite[Subsection 2.3]{bcw:gaaoe} or \cite[Theorem 2.18]{r:tgatlpa}).    Thus we may form the Steinberg algebra $A_K(\mathcal{G}_E)$.  

The key result for us is that the map 
\[\pi_E: L_K(E)\longrightarrow A_K(\mathcal{G}_E),   \]  
determined by $\pi_E(v) = 1_{Z(v, v)}$, $\pi_E(e) = 1_{Z(e, r(e))}$, and $\pi_E(e^*) = 1_{Z(r(e), e)},$
 is an algebra isomorphism;  see e.g.  \cite[Example 2.3]{cs:eghmesa} for the details of the argument.    In other words, every Leavitt path algebra is a Steinberg algebra.    
In particular,  we have $ \pi_E(\alpha\beta^* - \sum_{e\in F}\alpha ee^*\beta^*) = 1_{Z(\alpha, \beta, F)}.$

\begin{defn}\label{GEsubHdefn}
Let $E= (E^0, E^1, r, s)$ be an arbitrary graph, and $H$ a nonempty subset of $E^0$. We define the subset  $\mathcal{G}_E|_H$ of $\mathcal{G}_E$  by setting  $$\mathcal{G}_E|_H =\{(\alpha x, |\alpha| - |\beta|, \beta x)\in \mathcal{G}_E\mid  x\in X_E, s(\alpha), s(\beta) \in H\}.$$
\end{defn}

\begin{thm}\label{gecorthem}
Let $K$ be a field, $E= (E^0, E^1, r, s)$  an arbitrary graph, and $H$ a nonempty subset of $E^0$. 

$(1)$ $\mathcal{G}_E|_H$ is an open ample subgroupoid of $\mathcal{G}_E$.  

$(2)$ $\sum_{v, w\in H}vL_K(E)w\cong A_K(\mathcal{G}_E|_H)$.

$(3)$ If $H$ is finite, then $(\sum_{v\in H}v)L_K(E)(\sum_{v\in H}v)\cong A_K(\mathcal{G}_E|_H)$.
\end{thm}
\begin{proof}
(1) Since the sets $Z(\alpha, \alpha, F)$, where $\alpha\in E^*$ and $F \subseteq s^{-1}(r(\alpha))$ is finite, is a basis of compact open sets for the Hausdorff space $\mathcal{G}_E^{(0)}$, the set $U:= \bigcup_{v\in H}Z(v, v)$ is also an open subset of $\mathcal{G}_E^{(0)}$.	Also, we then have that for each $(\alpha x, |\alpha| - |\beta|, \beta x)\in \mathcal{G}_E$, $(\alpha x, |\alpha| - |\beta|, \beta x)\in r^{-1}_{\mathcal{G}_E}(U)\cap s^{-1}_{\mathcal{G}_E}(U)$ $\Longleftrightarrow$ $(\alpha x, 0, \alpha x)\in U$ and $(\beta x, 0, \beta x)\in U$ $\Longleftrightarrow$ $s(\alpha),\, s(\beta)\in H$. This implies that $r^{-1}_{\mathcal{G}_E}(U)\cap s^{-1}_{\mathcal{G}_E}(U) = \mathcal{G}_E|_H$, so $\mathcal{G}_E|_H$ is an open ample subgroupoid of $\mathcal{G}_E$ by Proposition \ref{cornerStein} (1).

(2)  By using the aforementioned properties of the isomorphism $\pi_E: L_K(E) \to A_K(\mathcal{G}_E)$  described in  \cite[Example 2.3]{cs:eghmesa}, the restriction of $\pi_E$ gives an isomorphism between these two corner-like algebras: 
$$\sum_{v, w\in H}vL_K(E)w\cong \sum_{v, w\in H}1_{Z(v, v)}A_K(\mathcal{G}_E)1_{Z(w, w)}.$$
We note that $Z(v, v)\cap Z(w, w) = \emptyset$ for all $v\neq w$,
and by item (1), $r^{-1}_{\mathcal{G}_E}(U)\cap s^{-1}_{\mathcal{G}_E}(U) = \mathcal{G}_E|_H$, where $U:= \bigcup_{v\in H}Z(v, v)$. Then, by Proposition \ref{cornerStein} (2), $$\sum_{v, w\in H}1_{Z(v, v)}A_K(\mathcal{G}_E)1_{Z(w, w)} = A_K(\mathcal{G}_E|_H),$$
so $$\sum_{v, w\in H}vL_K(E)w\cong A_K(\mathcal{G}_E|_H).$$

(3) Since $H$ is finite, we have $$(\sum_{v\in H}v)L_K(E)(\sum_{v\in H}v) = \sum_{v, w\in H}vL_K(E)w,$$ so 
$(\sum_{v\in H}v)L_K(E)(\sum_{v\in H}v) = A_K(\mathcal{G}_E|_H)$ by item (2), finishing the proof.
\end{proof}


In addition to providing the previously-promised instance of a corner of a Leavitt path algebra which is not a Leavitt path algebra, we illustrate many of the ideas that have been presented in this section in the following specific example.


\begin{exas}\label{cornernotleav}
Let $K$ be any field and $C= (C^0, C^1, r, s)$ the graph with $C^0 = \{v, w_n\mid n\in \mathbb{N}\}$, $C^1 =\{e_n \mid n\in \mathbb{N}\}$ and $r(e_n) = w_n$, $s(e_n) = v$ for all $n$.   (So $C$ is the ``infinite clock" graph described in \cite[Example 1.6.12]{AAS}.)   

Then $v L_K(C)v$ is not isomorphic to a Leavitt path algebra with coefficients in $K$ for any graph $F$, as follows.  
Since $(e_ne^*_n) (e_me^*_m) = \delta_{m, n}e_ne^*_n$ for all $n\in \mathbb{N}$, it is not hard to see that $vL_K(C)v = \text{Span}_K \{v, e_ne^*_n\mid n\in \mathbb{N}\}$.  Thus $vL_K(C)v$ is an infinite-dimensional commutative unital $K$-algebra.  So if $vL_K(C)v \cong  L_K(F)$ for some graph $F$, then by the unital property necessarily $F^0$ would be finite, and by commutativity $F$ would have  only isolated vertices and/or vertices with exactly one loop based at that vertex.  This would force $L_K(F)$ to be a finite direct sum of copies of $K$ and $K[x,x^{-1}]$.   But such an algebra contains only finitely many idempotents, and so can not be isomorphic to  $vL_K(C)v$.

	
	
	

Now define $H = \{v\}$.   Then  Theorem \ref{gecorthem} (3)  gives that   $vL_K(C)v$ is isomorphic to a Steinberg algebra, specifically $A_K(\mathcal{G}_{C|\{v\}})$, which we describe here in the notation of that result.   $C^*$ denotes  the set of all finite paths in $C$ (including vertices), so that  $C^* = \{v, w_n, e_n \ | \ n \in \mathbb{N}\}$.   $C^{\infty}$ denotes the set of all infinite paths in $C$, so that $C^\infty = \emptyset$.  $X_C $ is the union of $ C^{\infty} $ with the set of finite paths in $C$ that end in a singular vertex, so that  $X_C = \{v, w_n, e_n \  | \ n \in \mathbb{N}\}.$  Thus the groupoid $\mathcal{G}_C$ is explicitly described as  
$$\mathcal{G}_C := \{(\alpha x, |\alpha| - |\beta|, \beta x) \ | \  \alpha, \beta \in C^*, x \in X_C, r(\alpha) = s(x) = r(\beta)\} $$ 
$$  \ \ \ \ \ \ \ = \{(v, 0, v); (w_n, 0, w_n); (e_n, 0, e_n); (e_n, 1, w_n); (w_n, -1, e_n)  \ | \ n \in \mathbb{N}\}. $$
A description of the subgroupoid $ \mathcal{G}_{C|\{v\}}$ of $\mathcal{G}_{C}$ is given by
  $$\mathcal{G}_{C|\{v\}} := \{ (\alpha x, |\alpha| - |\beta|, \beta x) \in \mathcal{G}_C \  | \  s(\alpha) = v = s(\beta) \}  = \{(v, 0, v); (e_n, 0, e_n) \ | \ n \in \mathbb{N}\}.$$  
The groupoid $\mathcal{G}_{C|\{v\}}$    is ample,  with a basis of compact open sets consisting of sets of the form $Z(e_n, e_n)  $ and  $Z(v, v, F)$, namely, sets of the form    $\{(e_n, 0, e_n)\}$ $(n \in \mathbb{N})$ and $\{(v,0,v), (e_m, 0, e_m) \ | \ e_m \notin F\}$, where $F$ is any finite subset of $C^1$. 

The isomorphism between $vL_K(C)v$ and $A_K(\mathcal{G}_{C|\{v\}})$ is clear:   it is the $K$-linear extension of the map which takes $v$ to $1_{\mathcal{G}_{C|\{v\}}}$, and which takes  
$e_ne_n^*$ to $1_{\{(e_n,0,e_n)\}}$ for all $n\in \mathbb{N}$.   

  More concretely, it is easy to see that the unital algebra $vL_K(C)v \cong A_K(\mathcal{G}_{C|\{v\}})$ is isomorphic to the $K$-unital extension of the nonunital $K$-algebra 
$\oplus_{n\in \mathbb{N}}K$;  that is, the direct sum of countably infinitely many copies of $K$, with an extra copy of $K$ appended to provide a unit element. Namely,
let $A = K \times \oplus_{n\in \mathbb{N}}K $ be the unital algebra with component-wise addition an the multiplication defined by $(k,x) \cdot (k', x') = (k k', kx'+ k'x + x x'),$ 
$k,k' \in K,  x, x' \in     \oplus_{n\in \mathbb{N}}K .$ Then the map $ v \mapsto (1,0), e_n e^\ast _n \mapsto (0, \varepsilon _n),$ where $\{\varepsilon _n, n \in \mathbb{N}\}$ is the canonical  basis
of $ \oplus_{n\in \mathbb{N}}K ,$ determines an isomorphism of algebras.


\end{exas}

We close this section with the following useful corollary.

\begin{cor}\label{limofcor}
Let $\mathcal{E} = ((E_i)_{i\in I}, (\phi_{ij})_{i\leq j\ {\rm in}\ I})$ be a direct system in {\rm\textbf{CKGr}} and let $E$ be the direct limit for $\mathcal{E}$	in {\rm\textbf{CKGr}} with canonical morphisms $\eta_i: E_i\longrightarrow E$. For each $i\in I$ let $T_i$ be a finite subset of $E^0_i$ with the condition that $\phi_{ij}^0(T_i) \subseteq T_j$ for all $i\leq j$. Let $T := \bigcup_{i\in I}\eta_i^0(T_i) \subseteq E^0$. Let $K$ be any field. Then, $A_K(\mathcal{G}_E|_T)$ is a direct limit
for the system $((\sum_{v\in T_i}v)L_K(E_i)(\sum_{v\in T_i}v))_{i\in I}$ in $K$-{\rm\textbf{Alg}}.
\end{cor}
\begin{proof}
By Proposition \ref{Lpaproperties} (3), every CK-morphism $\phi_{ij}: E_i\longrightarrow E_j$ induces an injective $K$-algebra homomorphism $L_K(\phi_{ij}): L_K(E_i)\longrightarrow L_K(E_j)$ defined by $L_K(\phi_{ij})(v) = \phi^0_{ij}(v)$ for all $v\in E^0_i$ and $L_K(\phi_{ij})(e) = \phi_{ij}^1(e)$ for all $e\in E^1_i$. Then, for any $i\leq j$ in $I$, since $\phi_{ij}^0(T_i) \subseteq T_j$, $L_K(\phi_{ij})$ induces an injective $K$-algebra homomorphism
$$\psi_{ij}: (\sum_{v\in T_i}v)L_K(E_i)(\sum_{v\in T_i}v)\longrightarrow (\sum_{v\in T_j}v)L_K(E_j)(\sum_{v\in T_j}v),$$ so $((\sum_{v\in T_i}v)L_K(E_i)(\sum_{v\in T_i}v))_{i\in I}, (\psi_{ij})_{i\leq j\ {\rm in}\ I})$ is a direct system in $K$-\textbf{Alg}. 

Let $A$ be the direct limit for this system in $K$-\textbf{Alg} with canonical homomorphisms $\theta_i: (\sum_{v\in T_i}v)L_K(E_i)(\sum_{v\in T_i}v)\longrightarrow A$. We show that $A$ is isomorphic to the $K$-algebra $\sum_{v, w\in T}vL_K(E)w$. Indeed, by Proposition \ref{Lpaproperties} (2), $L_K(E)$ is the direct limit for the system $((L_K(E_i))_{i\in I}, (L_K(\phi_{ij}))_{i\leq j\ {\rm in}\ I})$ in $K$-\textbf{Alg} with canonical homomorphisms $L_K(\eta_i): L_K(E_i)\longrightarrow L_K(E)$ defined by $L_K(\eta_i)(v) = \eta_i^0(v)$ for all $v\in E^0_i$ and $L_K(\eta_i)(e) = \eta_i^1(e)$ for all $e\in E^1_i$. We note that 
$L_K(\eta_i)$ is injective for all $i\in I$, by Proposition \ref{Lpaproperties} (2). For each $i\in I$, since $\eta^0_i(T_i)\subseteq T$, $L_K(\eta_i)$ induces an injective $K$-algebra homomorphism $\lambda_i: (\sum_{v\in T_i}v)L_K(E_i)(\sum_{v\in T_i}v)\longrightarrow \sum_{v, w\in T}vL_K(E)w$. Since $L_K(\eta_j)L_K(\phi_{ij}) = L_K(\eta_i)$ for all $i\leq j$ in $I$, we have that $\lambda_j\psi_{ij} =\lambda_i$ for all $i\leq j$ in $I$, so there is a unique $K$-algebra homomorphism $\pi: A\longrightarrow \sum_{v, w\in T}vL_K(E)w$ such that $\pi\theta_i = \lambda_i$ for all $i\in I$. We show that
$\pi$ is an isomorphism.

Let $x \in \sum_{v, w\in T}vL_K(E)w$. By Remark \ref{oneedge&cor} (2), we have that $x$ may be written of the form $x = \sum^n_{k =1}r_kp_kq^*_k$, where $r_k\in K$, $p_k$ and $q_k$ $(1\leq k\leq n)$ are paths in $E$ such that $r_E(p_k)= r_E(q_k)$ and $s_E(p_k), s_E(q_k)\in T$. By Proposition \ref{Lpaproperties} (1), $E^0$ and $E^1$ are the direct limits of the corresponding direct systems of sets, $((E^0_i)_{i\in I}, (\phi^0_{ij})_{i\leq j\ {\rm in}\ I})$ and $((E^1_i)_{i\in I}, (\phi^1_{ij})_{i\leq j\ {\rm in}\ I})$, with canonical maps $\eta^0_i: E^0_i\longrightarrow E^0$ and $\eta^1_i: E^1_i\longrightarrow E^1$. Then, there exists $i\in I$ such that $p_k, q_k$ are paths in $\eta_i(E_i)$ and $s_E(p_k), s_E(q_k) \in \eta^0_i(T_i)$ for all $1\leq k\leq n$, whence $x\in {\rm Im}(\lambda_i)\subseteq \pi(A)$. Consequently, $\pi$ is surjective.
Now consider $a\in \ker(\pi)$, write $a = \theta_i(b)$ for some $i\in I$ and $b\in (\sum_{v\in T_i}v)L_K(E_i)(\sum_{v\in T_i}v)$. We then have that $\lambda_i(b) = \pi\theta_i(b) = \pi(\theta_i(b))=\pi(a) = 0$. Since $\lambda_i$ is injective, $b = 0$, so $a = \theta_i(0) =0$. Thus $\pi$ is an isomorphism, as announced.

Finally, by Theorem \ref{gecorthem} (2), we have that $\sum_{v, w\in T}vL_K(E)w \cong A_K(\mathcal{G}_E|_T)$ as $K$-algebras, whence $\varinjlim_I (\sum_{v\in T_i}v)L_K(E_i)(\sum_{v\in T_i}v) \cong A_K(\mathcal{G}_E|_T)$, thus finishing the proof.
\end{proof}

\medskip

%
%
%
%

\section{Endomorphism rings of finitely generated projective modules over Leavitt path algebras of arbitrary graphs}
The main goal of this section   is to show that the endomorphism ring of a nonzero finitely generated projective module over the Leavitt path algebra of an arbitrary graph is isomorphic to a Steinberg algebra (Theorem \ref{End(Q)general}).  Consequently, we get that every algebra with local units over a given field which is Morita equivalent to the Leavitt path algebra of an arbitrary row-countable graph is indeed isomorphic to a Steinberg algebra (Theorem \ref{moquiLeavitt}).

\begin{lem}\label{projmod}
Let $E= (E^0, E^1, r, s)$ be an arbitrary graph and $K$ a field. Then every nonzero finitely generated projective left $L_K(E)$-module $Q$ may be written in the form $$Q\cong (\bigoplus_{v\in V} n_v L_K(E)(v - \sum_{e\in T_v}ee^*)) \oplus (\bigoplus_{w\in W} n_w L_K(E)w),$$ where $V$ and $W$ are finite subsets of $E^0$, each $v\in V$ is an infinite emitter, each $T_v$ is a nonempty finite subset of $s^{-1}(v)$, and the numbers $n_v, n_w$ are positive integers.
\end{lem}
\begin{proof} Let $Q$ be a nonzero finitely generated projective left $L_K(E)$-module. By Theorem \ref{AG-HLMRTthm} (3), there exist a nonempty finite subset $S\subseteq E^0\times 2^{E^1}$ and positive integers $\{n(v, T)\}_{(v, T)\in S}$ such that
\begin{itemize}	
\item[(1)] $Q\cong \bigoplus_{(v, T)\in S} n(v, T) L_K(E)(v- \sum_{e\in T}ee^*)$;
\item[(2)] For all $(v, T)\in S$, $T$ is a finite subset of $s^{-1}(v)$;
\item[(3)] For all $(v, T)\in S$, $T$ is nonempty only if $v$ is an infinite emitter.
\end{itemize}
For any infinite emitter $v$ with $(v, T)\in S$ for some nonempty subset $T\subseteq E^1$, we denote $T_v := \bigcup_{(v, T)\in S}T\subset s^{-1}(v)$. Then, by Theorem \ref{AG-HLMRTthm} (2), for each $(v, T) \in S$ with $T\neq \emptyset$, we have that 
\begin{equation*}
L_K(E)(v- \sum_{e\in T}ee^*)\cong L_K(E)(v- \sum_{e\in T_v}ee^*) \bigoplus (\bigoplus_{e\in T_v\setminus T}L_K(E)r(e)).  \tag{\mbox{$\ast$}}
\end{equation*}
Now replace any one of the summands isomorphic to $L_K(E)(v- \sum_{e\in T}ee^*)$, where $T$ is a nonempty finite subset of $s^{-1}(v)$, which appears in the  decomposition (1)  of $Q$  by the isomorphic version of $L_K(E)(v- \sum_{e\in T}ee^*)$ given in $(\ast)$. Continuing this  process on all such other vertices $v$, we get a direct sum decomposition of $Q$ as in the statement, finishing the proof. 
\end{proof}

\begin{defn}[{\cite[Definition 2.6]{aalp:tcqflpa}}: the ``out-split" graph]\label{outsplitdef}
	Let $E = (E^0, E^1, r, s)$ be a graph and $v\in E^0$ a vertex that is not a sink. Partition $s^{-1}(v)$ into a finite number, say $n$, of disjoint nonempty subsets $\mathcal{E}_1, \mathcal{E}_2, ..., \mathcal{E}_n$.  We form the \textit{out-split graph} $E_{os} = (E^0_{os}, E^1_{os}, r_{os}, s_{os})$ from $E$ using the partition $\{\mathcal{E}_i\mid i = 1, ..., n\}$ as follows:
	$E^0_{os} = (E^0\setminus \{v\}) \cup \{v^1, v^2, ..., v^n\}$,
	
	\begin{center}
		$E^1_{os} = \{e^1, e^2, ..., e^n\mid e\in E^1, r(e) = v\} \cup \{f\mid f\in E^1\setminus r^{-1}(v)\}$,
	\end{center} 
	and define $r_{os}$, $s_{os}: E_{os}^1 \longrightarrow E_{os}^0$ by setting	$r_{os}(e^j) = v^j$,  $r_{os}(f) = r(f)$, 	and
	
	\begin{equation*}
	s_{os}(x)=  \left\{
	\begin{array}{lcl}
	s(f)&  & \text{if } x = f\notin s^{-1}(v)\\
	v^i&  & \text{if } x =f\in s^{-1}(v) \text{ and } f\in \mathcal{E}_i\\
	s(e)&  & \text{if } x = e^j \text{ and } e\notin s^{-1}(v)\\
	v^i&  & \text{if } x = e^j,\ e\in s^{-1}(v) \text{ and } e\in \mathcal{E}_i
	\end{array}%
	.
	\right.
	\end{equation*}%
\end{defn}


The following proposition can be seen as an extension of \cite[Theorem 2.8]{aalp:tcqflpa}.

\begin{prop} \label{outsplitprop}
Let $K$ be any field. Let $E$ be an arbitrary graph, $v\in E^0$ a  vertex which is not a sink, and a partition \[s^{-1}(v) = \mathcal{E}_1 \sqcup \mathcal{E}_2 \sqcup \hdots \sqcup \mathcal{E}_n\] is chosen with at most one of the $\mathcal{E}_i$ is infinite. Then $L_K(E)\cong L_K(E_{os})$ as $\mathbb{Z}$-graded $K$-algebras. This isomorphism yields an isomorphism of categories
\begin{center}
$\Phi: L_K(E)\text{-Mod} \longrightarrow L_K(E_{os})\text{-Mod}$
\end{center}
for which $\Phi(L_K(E)v) = \oplus^n_{i=1}L_K(E_{os})v^i$ and $\Phi(L_K(E)w) = L_K(E_{os})w$ for all $w\in E^0\setminus \{v\}$.
\end{prop}
\begin{proof} The quoted result \cite[Theorem 2.8]{aalp:tcqflpa} applies to constructions more general than the out-split construction of row-finite graphs. Accordingly, based on the proof of \cite[Theorem 3.2]{bp:feoga}, we provide here a short proof of Proposition \ref{outsplitprop}.
	
We define the  elements $\{Q_u \ | \ u\in E^0\}$ and $\{T_e, T_{e^*} \ | \ e\in E^1\}$ of $L_K(E_{os})$  by setting 
\begin{equation*}
Q_u=  \left\{
\begin{array}{lcl}
\sum^n_{i=1}v^i&  & \text{if } u = v , \\
u&  & \text{otherwise \ \ ,}%
\end{array}%
\right.
\end{equation*}%
\medskip
\begin{equation*}
T_e=  \left\{
\begin{array}{lcl}
\sum_{i=1}^n e^i&  & \text{if } e\in r^{-1}(v)\\
e&  & \text{otherwise}  \ \ , \ \ \ \ \   
\end{array}%
\right.
\end{equation*}%
\medskip
and 
\begin{equation*}
T_{e^*}=  \left\{
\begin{array}{lcl}
\sum_{i=1}^n (e^i)^*&  & \text{if } e\in r^{-1}(v)\\
e^*&  & \text{otherwise}  \ \ .
\end{array}%
\right.
\end{equation*}%
\medskip

\noindent
By repeating verbatim the corresponding argument in
the proof of \cite[Theorem 2.8]{aalp:tcqflpa}, we get that $\{Q_u, T_e, T_{e^*}\mid u\in E^0, e\in E^1\}$ is a family in $L_K(E_{os})$ satisfying the same relations as $\{u, e, e^*\mid u\in E^0, e\in E^1\}$. Then, by the Universal Homomorphism Property of $L_K(E)$, there exists a $K$-algebra homomorphism $\pi: L_K(E)\longrightarrow L_K(E_{os})$, which maps $u\longmapsto Q_u$, $e\longmapsto T_e$ and $e^*\longmapsto T_{e^*}$.
Since $Q_u$ has degree $0$, $T_e$ has degree $1$, and $T_{e^*}$ has degree $-1$ for all $u\in E^0$ and $e\in E^1$, $\pi$ is thus a $\mathbb{Z}$-graded homomorphism, whence the injectivity of $\pi$ is guaranteed by \cite[Theorem 4.8]{tomf:utaisflpa}. To prove that $\pi$ is surjective, we show that the generators of $L_K(E_{os})$ lie in $Im(\pi)$. If $w \in E_{os}^0\setminus \{v^1, v^2, ..., v^n\}$, then $w = \pi(w) \in Im(\pi)$. Since $\{\mathcal{E}_i\}_{i=1}^n$ are chosen with at most one of the $\mathcal{E}_i$  infinite, there is a unique number $i$ such that $v^i$ is an infinite emitter. Without loss of generality we may assume that $i = 1$, and so $v^j$ is a regular vertex for all $2\leq j\leq n$. For each $2\leq j\leq n$, we have 
$$v^j = \sum_{f\in s_{os}^{-1}(v^j)} ff^* = \sum_{e\in \mathcal{E}_j}T_eT_{e^*} = \sum_{e\in \mathcal{E}_j}\pi(ee^*)\in Im(\pi)$$ and $$v^1 = Q_v - \sum_{j=2}^nv^j = \pi(v) - \sum_{j=2}^nv^j \in Im(\pi).$$

If $f \in  E_{os}^1\setminus \{e^1, e^2, ..., e^n\mid e\in E^1, r(e) = v\}$, then $f = \pi(f) \in Im(\pi)$. For each $e\in r^{-1}(v)$, we have $\pi(e) = T_e = \sum_{i=1}^n e^i$, so $e^i = \pi(e)v^i \in Im(\pi)$ for all $i = 1, ..., n$. Therefore, $\pi$ is surjective, so it is an isomorphism. Moreover, this  isomorphism maps $w$ to $\sum_{i=1}^nv^i$ if $w=v$, and to $w$ otherwise, so that the associated isomorphism of categories
restricts to the desired map, finishing the proof.
\end{proof}

Consequently, we get the following useful result.
\begin{cor}\label{outsplitcor}
Let $K$ be a field, $E$ an arbitrary graph, $v$ an infinite emitter and $T_v$ a nonempty finite subset of $s^{-1}(v)$. Put $\mathcal{E}_1 = T_v$ and $\mathcal{E}_2 = s^{-1}(v)\setminus T_v$. Let $F$ be the graph obtained by out-splitting the vertex $v$ into
the vertices $v^1, v^2$ according to  the partition $\mathcal{E}_1, \mathcal{E}_2$. Then $L_K(E)\cong L_K(F)$ as $\mathbb{Z}$-graded $K$-algebras. This isomorphism yields an isomorphism of categories
\begin{center}
$\Phi_v: L_K(E){\rm-Mod} \longrightarrow L_K(F){\rm-Mod}$
\end{center}
with the following properties:
\begin{itemize}	
\item[(1)] $\Phi_v(L_K(E)v) = L_K(F)v^1\oplus L_K(F)v^2$ and $\Phi_v(L_K(E)w) = L_K(F)w$ for all $w\in E^0\setminus \{v\}$,
\item[(2)] For all $w\in E^0\setminus \{v\}$ and all finite subsets $W\subseteq s^{-1}(w)$ with $v\notin r_E(W)$,
$$\Phi_v(L_K(E)(w - \sum_{e\in W}ee^*)) = L_K(F)(w - \sum_{e\in W}ee^*),$$
\item[(3)] For all $w\in E^0\setminus \{v\}$ and any finite subset $W\subseteq s^{-1}(w)$ with $v\in r_E(W)$, there
exists a finite subset $W'\subseteq s_F^{-1}(w)$ such that
$$\Phi_v(L_K(E)(w - \sum_{e\in W}ee^*)) = L_K(F)(w - \sum_{e\in W'}ee^*),$$
\item[(4)] $\Phi_v(L_K(E)(v - \sum_{e\in T_v}ee^*)) = L_K(F)v^2$.
\end{itemize}
\end{cor}
\begin{proof}
By Proposition \ref{outsplitprop}, the map $\pi_v: L_K(E)\longrightarrow L_K(F)$, defined by 
\begin{equation*}
\pi_v(u)=  \left\{
\begin{array}{lcl}
v^1 + v^2&  & \text{if } u = v , \\
u&  & \text{otherwise \ \ ,}%
\end{array}%
\right.
\end{equation*}%
\medskip
\begin{equation*}
\pi_v(e)=  \left\{
\begin{array}{lcl}
e^1 + e^2&  & \text{if } e\in r^{-1}(v)\\
e&  & \text{otherwise}  \ \ , \ \ \ \ \   
\end{array}%
\right.
\end{equation*}%
\medskip
and 
\begin{equation*}
\pi_v(e^*)=  \left\{
\begin{array}{lcl}
(e^1)^* + (e^2)^*&  & \text{if } e\in r^{-1}(v)\\
e^*&  & \text{otherwise }, 
\end{array}%
\right.
\end{equation*}%
extends to an isomorphism of $\mathbb{Z}$-graded $K$-algebras, and this isomorphism yields an isomorphism of categories
\begin{center}
	$\Phi_v: L_K(E)\text{-Mod} \longrightarrow L_K(F)\text{-Mod}$
\end{center}
with property (1) of the statement.

(2) Let $w\in E^0\setminus \{v\}$ and $W$ a finite subset of $s^{-1}(w)$ with $v\notin r_E(W)$. We then have that $\pi_v(w - \sum_{e\in W}ee^*) = w - \sum_{e\in W}ee^*$, so $\Phi_v(L_K(E)(w - \sum_{e\in W}ee^*)) = L_K(F)(w - \sum_{e\in W}ee^*)$.

(3) Let $w\in E^0\setminus \{v\}$ and $W$ a finite subset of $s^{-1}(w)$ with $v\in r_E(W)$. Put $H = \{e\in W\mid r(e) = v\}$. We then for $e\in H$ have that $e^1(e^2)^* = (e^1v^1)(v^2(e^2)^*) = e^1 (v^1v^2)(e^2)^* = 0$ and $e^2(e^1)^* = (e^2v^2) (v^1 ((e^1)^*)) = e^2(v^2v^1)(e^1)^* = 0$, and
\begin{equation*}
\begin{array}{rcl}
\pi_v(w - \sum_{e\in W}ee^*) &=& w - \sum_{e\in H}(e^1 + e^2)((e^1)^* + (e^2)^*) - \sum_{e\in W\setminus H}ee^*\\
&=& w - \sum_{e\in H}(e^1(e^1)^* + e^2(e^2)^*) - \sum_{e\in W\setminus H}ee^*\\
&=& w - \sum_{e\in W'}ee^*,
\end{array}
\end{equation*}
where $W' = (W\setminus H)\, \sqcup\, \{e^1, e^2\mid e\in H\}\subseteq s_F^{-1}(w)$. This implies that $\Phi_v(L_K(E)(w - \sum_{e\in W}ee^*)) = L_K(F)(w - \sum_{e\in W'}ee^*)$.

(4) Let $H = \{e\in T_v\mid r(e) = v\}$. Then, similar to item (3) we have that $e^1(e^2)^* = 0 = e^2(e^1)^*$ and
$\pi_v(v - \sum_{e\in T_v}ee^*) = v^1 + v^2 - \sum_{e\in H}(e^1 + e^2)((e^1)^* + (e^2)^*) - \sum_{e\in T_v\setminus H}ee^* = v^1 + v^2 - \sum_{e\in H}(e^1(e^1)^* + e^2(e^2)^*) - \sum_{e\in T_v\setminus H}ee^*$. 

On the other hand, we have that $s_F^{-1}(v^1) = (T_v\setminus H) \sqcup \{e^1, e^2\mid e\in H\}$, so $v^1 = \sum_{e\in H}(e^1(e^1)^* + e^2(e^2)^*) + \sum_{e\in T_v\setminus H}ee^*$ in $L_K(F)$. This implies that $\pi_v(v - \sum_{e\in T_v}ee^*) = v^1 + v^2 - \sum_{e\in H}(e^1(e^1)^* + e^2(e^2)^*) - \sum_{e\in T_v\setminus H}ee^* = v^2$, and hence $\Phi_v(L_K(E)(v - \sum_{e\in T_v}ee^*)) = L_K(F)v^2$, finishing the proof.
\end{proof}	

Using Lemma \ref{projmod} and Corollary \ref{outsplitcor} we get the following result which plays an important role in the proof of the main theorem below.

\begin{prop}\label{cateiso}
Let $K$ be a field, $E$ an arbitrary graph and $Q$ a nonzero finitely generated projective left $L_K(E)$-module. 	Then there exists a graph $F$ with the following properties.
\begin{itemize}	
\item[(1)] $F$ is obtained from $E$ in some step-by-step process of out-splittings.  

\item[(2)] There exists an isomorphism of categories
\begin{center}
	$\Phi: L_K(E){\rm -Mod} \longrightarrow L_K(F){\rm -Mod}$
\end{center}
such that $\Phi(Q) \cong \bigoplus_{v\in H}n_v L_K(F)v$ for some finite subset $H\subseteq F^0$ and some positive integers $\{n_v\}_{v\in H}$.
\end{itemize}
\end{prop}
\begin{proof} By Lemma \ref{projmod}, we have that
$$Q\cong (\bigoplus_{v\in V} n_v L_K(E)(v - \sum_{e\in T_v}ee^*) ) \oplus (\bigoplus_{w\in W} m_w L_K(E)w),$$ where $V$ and $W$ are some finite subsets of $E^0$, each $v\in V$ is an infinite emitter, each $T_v$ is some nonempty finite subset of $s^{-1}(v)$, and  $n_v, m_w$ are some positive integers.	Write $V = \{v_1, v_2, \hdots, v_n\}$. Let $E_1$ be the graph obtained from $E$ by out-splitting the vertex $v_1$ into the vertices $v_1^1, v_1^2$ according to the partition $\mathcal{E}_1 = T_{v_1}$, $\mathcal{E}_2 = s^{-1}(v_1)\setminus T_{v_1}$. By Corollary \ref{outsplitcor}, there exists an isomorphism of categories
\begin{center}
$\Phi_{v_1}: L_K(E){\rm-Mod} \longrightarrow L_K(E_1){\rm-Mod}$
\end{center}
with the following properties:
\begin{itemize}	
\item[(i)] $\Phi_{v_1}(L_K(E)v_1) = L_K(E_1 )v_1^1\oplus L_K(E_1)v_1^2$ and $\Phi_{v_1}(L_K(E)w) = L_K(E_1)w$ for all $w\in E^0\setminus \{v_1\}$,
\item[(ii)] for all $w\in E^0\setminus \{v_1\}$ and all finite subsets $W\subseteq s^{-1}(w)$ with $v_1\notin r_E(W)$,
$$\Phi_{v_1}(L_K(E)(w - \sum_{e\in W}ee^*)) = L_K(E_1)(w - \sum_{e\in W}ee^*),$$
\item[(iii)] for all $w\in E^0\setminus \{v_1\}$ and any finite subset $W\subseteq s^{-1}(w)$ with $v\in r_E(W)$, there
exists a finite subset $W'\subseteq s_{E_1}^{-1}(w)$ such that
$$\Phi_{v_1}(L_K(E)(w - \sum_{e\in W}ee^*)) = L_K(E_1)(w - \sum_{e\in W'}ee^*), \ \mbox{and}$$
\item[(iv)] $\Phi_{v_1}(L_K(E)(v_1 - \sum_{e\in T_{v_1}}ee^*)) = L_K(E_1)v_1^2$.
\end{itemize}
We then have that $\Phi_{v_1}(Q)$ is isomorphic to 
$$(\bigoplus^n_{i=2}n_{v_i} L_K(E_1)(v_i - \sum_{e\in T'_{v_i}}ee^*))\oplus n_{v_1}L_K(E_1)v_1^2 \oplus (\bigoplus_{w\in W''} m''_w L_K(E_1)w)$$ as left $L_K(E_1)$-modules, where $T'_{v_i}$ is some nonempty finite subset of $s^{-1}_{E_1}(v_i)$, and $W'', m''_w$ are defined by setting $$W'' = W \mbox{ if } v_1\notin W $$
\noindent
 (in this case we have that $ m''_w = m_w$  for all $ w\in W$), and 
$$W'' = \{v^1_1, v^2_1\} \cup W\setminus\{v_1\} \mbox{ otherwise }$$
\noindent
 (in this case we have that $ m''_w = m_w  $ for all $ w\in W\setminus\{v_1\} $ and $m''_{v_1^i} = m_{v_1}).$ 

 Let $W' := W'' \cup \{v_1^2\}$, $m'_w := m''_w$ for all $w\in W'\setminus\{v^2_1\}$, and $m'_{v^2_1} := n_{v_1} + m''_{v_1^i}$ (note that $m''_{v_1^i} = 0$ if  $v_1\notin W$). We then get that
$$\Phi_{v_1}(Q)\cong (\bigoplus^n_{i=2}n_{v_i} L_K(E_1)(v_i - \sum_{e\in T'_{v_i}}ee^*))\oplus (\bigoplus_{w\in W'} m'_w L_K(E_1)w)$$ as left $L_K(E_1)$-modules.

We repeat the process described above, starting with the graph $E_2$ which is obtained from $E_1$ by out-splitting $v_2$ into $v_2^1$ and $v_2^2$ with respect to  $\mathcal{E}_1 = T'_{v_2}$, $\mathcal{E}_2 = s^{-1}_{E_1}(v_2)\setminus T'_{v_2}$. We see that after $n$ steps we arrive at the  graph $F$ of the statement and an isomorphism of categories 	$\Phi: L_K(E){\rm -Mod} \longrightarrow L_K(F){\rm -Mod}$ such that $\Phi(Q) \cong \bigoplus_{v\in H}n_v L_K(F)v$ for some finite subset $H\subseteq F^0$ and some positive integers $\{n_v\}_{v\in H}$, finishing the proof.
\end{proof}

We are now in position to achieve the  main result of this section.   

\begin{thm}\label{End(Q)general}
Let $K$ be any field, $E$ an arbitrary graph, and $Q$ a nonzero finitely generated projective left $L_K(E)$-module.   Then ${\rm End}_{L_K(E)}(Q)$ is isomorphic to a Steinberg  algebra.  
\end{thm}

\begin{proof}
By Proposition \ref{cateiso}, there exist a graph $F$ and an isomorphism of categories 	$\Phi: L_K(E){\rm -Mod} \longrightarrow L_K(F){\rm -Mod}$ such that $\Phi(Q) \cong \bigoplus_{v\in T}n_v L_K(F)v$ for some finite subset $T\subseteq F^0$ and some positive integers $\{n_v\}_{v\in T}$. The categorical isomorphism yields that ${\rm End}_{L_K(E)}(Q) \cong {\rm End}_{L_K(F)}(\Phi(Q))$.  Write $T = \{v_1, v_2, \dots , v_u\}$ and let $n_i := n_{v_i}$ for all $v_i\in T$.  Note that there are $\sigma = \sum_{i=1}^u n_i$ direct summands in the above decomposition of $\Phi(Q)$.  By  Proposition \ref{End}, ${\rm End}_{L_K(F)}(\Phi(Q))$ is isomorphic to a $\sigma \times \sigma$ matrix ring, with entries described as follows.  The indicated matrices   may be viewed as consisting of rectangular blocks of size $n_{i} \times n_{j}$, where, for $1 \leq i \leq u$, $1\leq j \leq u$, the entries of the $(i,j)$ block are elements of the $K$-vector space $v_i L_K(F) v_j$. 

Let $G$ be the graph formed from $F$ by taking each $v_i\in T$ and attaching a ``head" of length $n_i-1$ of the form
$$\xymatrix{  \bullet^{v_i^{n_i-1}}\ar[r]^{e_i^{n_i-1}} & \cdots \bullet^{v_i^2} \ar[r]^{e_i^2} & \bullet^{v_i^1} \ar[r]^{e_i^1} &\bullet^{v_i}}$$ to $F$. By construction,  $F$ is a CK-subgraph of $G$, so that by Proposition \ref{Lpaproperties} (3)  we may view $L_K(F)$ as a $K$-subalgebra of $L_K(G)$.

For each $1 \leq i \leq u$, and each $1\leq y \leq n_i-1$,  let $p_i^y:= e_i^y \cdots e_i^1$ denote the (unique) path in $G$ having  $s(p_i^y) = v_i^y$, and $r(p_i^y) = v_i$.    Note that, because of the specific configuration of the added vertices and edges used to build $G$ from $F$, repeated application of Remark \ref{oneedge&cor} (1) gives that  $p_i^y (p_i^y)^* = v_i^y$ in $L_K(G)$.
Let $H := T \cup \{v_i^j\mid 1\leq i\leq u, 1\leq j\leq n_i-1\}\subseteq G^0$. Note also that $|H| =   \sum_{1\leq i \leq u} n_i$, which  is precisely $\sigma$. Let $P := \bigoplus_{v\in H}L_K(G)v$ and again using Proposition \ref{End}, we obtain that ${\rm End}_{L_K(G)}(P)$ is isomorphic to the $\sigma \times \sigma$ matrix ring with entries described as follows.   For $1 \leq i, j \leq u$, and $0 \leq y \leq n_i-1$, $0 \leq z \leq n_j-1,$ the entries in the row indexed by $(n_i, y)$ and column indexed by $(n_j, z)$ are elements of $v_i^y L_K(G) v_j^z,$ where    $v_i^0=v_i.$

We now show that these two $\sigma \times \sigma$ matrix rings are isomorphic as $K$-algebras.  To do so, we show first that for each pair  $(n_i, y)$,  $(n_j, z)$ with $1 \leq i, j \leq u$, and $1 \leq y \leq n_i-1$, $1 \leq z \leq n_j-1$, there is a $K$-vector space isomorphism
$$\varphi = \varphi_{(n_i, y), (n_j,z)} :   v_i L_K(F) v_j \rightarrow v_i^y L_K(G) v_j^z.$$ 
      For $r
      \in L_K(F)$ we define
$$ \varphi_{(n_i, y), (n_j,z)}(v_i r v_j) = p_i^yv_i r v_j (p_j^z)^*.$$
Since $L_K(F)$ is a $K$-subalgebra of $L_K(G)$, we easily see that $\varphi$ is $K$-linear.   Further, $\varphi$ is a monomorphism:  if $p_i^yv_i r v_j (p_j^z)^* = 0$ then multiplying on the left by $(p_i^y)^*$ and on the right by $p_j^z$ yields $v_i r v_j = 0$.   To show  $\varphi$ is surjective:  for $v_i^y s v_j^z \in v_i^y L_K(G) v_j^z$ with $s\in L_K(G)$, define $s' = (p_i^y)^*v_i^y  s v_j^z p_j^z \in v_i L_K(G) v_j$.  But  using the fact that there are no paths in $G$ from elements of $T$ to any of the newly added vertices which yield $G$, we have as above that $s'$ may be viewed as an element of $L_K(F)$. Then, using the previous observation that $p_i^y (p_i^y)^* = v_i^y$ in $L_K(G)$, we conclude that $\varphi(s') = p_i^y (p_i^y)^* v_i^y s v_j^z p_j^z (p_j^z)^*  = v_i ^y s v_j^z$, and thus $\varphi$ is surjective.

We now define $\Theta$ to be the $K$-space isomorphism between the two matrix rings induced by applying each of the $\varphi_{(m_i, y), (m_j,z)}$ componentwise.   We need only show that these componentwise isomorphisms respect the corresponding matrix multiplications.    To do so, it suffices to show that the maps behave correctly in each component.    That is, we need only show, for each $n_\ell$ ($1\leq \ell \leq u$)   and each $x$ ($1 \leq x < n_\ell$), that 
$$ \varphi_{(m_i, y), (m_\ell,x)}(v_i r v_\ell) \cdot  \varphi_{(m_\ell, x), (m_j,z)}(v_\ell  r' v_j) =  \varphi_{(m_i, y), (m_j,z)}(v_i r v_\ell r'  v_j).$$ 
But this is immediate, as $(p_\ell^x)^* p_\ell^x = v_\ell$ for each $v_\ell \in T$ and $1 \leq x < n_\ell.$ Thus, we obtain that 
${\rm End}_{L_K(F)}(\Phi(Q))\cong {\rm End}_{L_K(G)}(P)$, and hence ${\rm End}_{L_K(E)}(Q)\cong {\rm End}_{L_K(G)}(P)= {\rm End}_{L_K(G)}(\bigoplus_{v\in H}L_K(G)v)\cong (\sum_{v\in H}v)L_K(G)(\sum_{v\in H}v)$ as $K$-algebras. Moreover, by Theorem \ref{gecorthem} (3), we have that $$(\sum_{v\in H}v)L_K(G)(\sum_{v\in H}v)\cong A_K(\mathcal{G}_G|_H)$$ as $K$-algebras, so ${\rm End}_{L_K(E)}(Q)\cong A_K(\mathcal{G}_G|_H)$, finishing the proof.
\end{proof}

We note that some of the techniques utilized in the proof of Theorem \ref{End(Q)general} were also used in establishing \cite[Theorem 3.8]{an:colpaofgalpa}.   

The following example will help illuminate the ideas of Theorem \ref{End(Q)general}.  

\begin{exas}\label{isoexample}
Let $E$ be the graph
$$E \ \ \ = \ \ \  \xymatrix{  \bullet^{v}\ar@(ul,ur)^{e} \ar[r]^{(\infty)}  &  \bullet^{w}\ar[r]^{(\infty)}&\bullet^{u}} $$
where there are infinitely many edges $\{f_i\mid i\in \mathbb{Z}^+\}$ from $v$ to $w.$  Consider as just one example the  nonzero finitely generated projective left $R$-module $Q = L_K(E)(v- ee^*)\oplus L_K(E)(v - f_1f^*_1)$. Then, by Lemma \ref{projmod}, we have $$Q\cong 2 L_K(E)(v- ee^* - f_1f^*_1)\oplus L_K(E)v\oplus L_K(E)w.$$ Let $T_v =\{e, f_1\}\subset s^{-1}(v)$ and let $F$ be the graph obtained from $E$ by out-splitting the vertex $v$ into the vertices $v_1, v_2$ according to the partition $\mathcal{E}_1:= T_v, \mathcal{E}_2:= s^{-1}(v)\setminus T_v$, pictured here: 
$$F \ \ \ = \ \ \  \xymatrix{  \bullet^{v_1}\ar@(ul,ur)^{e_1}\ar[d]_{e_2}\ar[r]^{f_1}&  \bullet^{w}\ar[r]^{(\infty)}&\bullet^{u}\\ \bullet^{v_2}\ar[ur]_{(\infty)}}.$$
By Corollary~\ref{outsplitcor}, there exists an isomorphism of categories $\Phi: L_K(E){\rm -Mod} \longrightarrow L_K(F){\rm -Mod}$ such that $\Phi(L_K(E)v) = L_K(F)v_1\oplus L_K(F)v_2$, $\Phi(L_K(E)w) = L_K(F)w$ and $\Phi(L_K(E)(v - ee^* - f_1f^*_1)) = L_K(F)v_2$, and hence
 \[\Phi(Q) \cong  2L_K(F)v_2 \oplus (L_K(F)v_1 \oplus L_K(F)v_2) \oplus L_K(F)w   \cong  L_K(F)v_1\oplus 3L_K(F)v_2\oplus L_K(F)w.\] 
 This decomposition dictates that we construct the graph $G$, graphically:
$$G  =  \xymatrix{&&  \bullet^{v_1}\ar@(ul,ur)^{e_1}\ar[d]_{e_2}\ar[r]^{f_1}&  \bullet^{w}\ar[r]^{(\infty)}&\bullet^{u}\\ \bullet^{v^2_2}\ar[r]&\bullet^{v^1_2}\ar[r]&\bullet^{v_2}\ar[ur]_{(\infty)}}.$$
Consider the finitely generated projective left $L_K(G)$-module
$$P := L_K(G)v_1\oplus L_K(G)v^2_2\oplus L_K(G)v^1_2\oplus L_K(G)v_2\oplus L_K(G)w.$$ By Theorem \ref{End(Q)general}, we have ${\rm End}_{L_K(F)}(\Phi(Q))\cong {\rm End}_{L_K(G)}(P)$. For notational simplification, let $R$ and $S$ denote $L_K(F)$ and $L_K(G)$, respectively.  Then   the explicit description of these two algebras as matrix rings as described in the proof of  Theorem \ref{End(Q)general} is:
\medskip
\footnotesize
$$   \left(\begin{tabular}{ccccc} 
$v_1Rv_1$&$v_1Rv_2$&$v_1Rv_2$&$v_1Rv_2$&$v_1Rw$\\
$v_2Rv_1$&$v_2Rv_2$&$v_2Rv_2$&$v_2Rv_2$&$v_2Rw$\\
$v_2Rv_1$&$v_2Rv_2$&$v_2Rv_2$&$v_2Rv_2$&$v_2Rw$\\
$v_2Rv_1$&$v_2Rv_2$&$v_2Rv_2$&$v_2Rv_2$&$v_2Rw$\\
$wRv_1$&$wRv_2$&$wRv_2$&$wRv_2$&$wRw$\\
\end{tabular}\right) 
\cong
 \left(\begin{tabular}{ccccc} 
$v_1Sv_1$&$v_1Sv^2_2$&$v_1Sv^1_2$&$v_1Sv_2$&$v_1Sw$\\
$v^2_2Sv_1$&$v^2_2Sv^2_2$&$v^2_2Sv^1_2$&$v^2_2Sv_2$&$v^2_2Sw$\\
$v^1_2Sv_1$&$v^1_2Sv^2_2$&$v^1_2Sv^1_2$&$v^1_2Sv_2$&$v^1_2Sw$\\
$v_2Sv_1$&$v_2Sv^2_2$&$v_2Sv^1_2$&$v_2Sv_2$&$v_2Sw$\\
$wSv_1$&$wSv^2_2$&$wSv^2_2$&$wSv_2$&$wSw$
\end{tabular}\right)  .$$

\normalsize
\noindent
Thus we get that \[{\rm End}_{L_K(E)}(Q)\cong {\rm End}_{L_K(G)}(P)\cong (\sum_{v\in H}v)L_K(G)(\sum_{v\in H}v)\cong A_K(\mathcal{G}_G|_H),\] where $H =\{v_1, v^2_2, v^1_2, v_2, w\}\subset G^0$.
\smallskip

\end{exas}

\begin{cor}\label{cornercor}
Let $K$ be any field, $E$ an arbitrary graph, and  $\varep$ a nonzero idempotent in $L_K(E)$.   Then the corner $\varep L_K(E) \varep$ of $L_K(E)$  is isomorphic to a Steinberg algebra.  
\end{cor}
\begin{proof} The statement follows from Theorem \ref{End(Q)general}, as $L_K(E)\varepsilon$ is a nonzero finitely generated projective left $L_K(E)$-module, and ${\rm End}_{L_K(E)}(L_K(E)\varepsilon) \cong \varepsilon L_K(E) \varepsilon.$   
\end{proof}

Let $R$ be a ring with local units. Recall (\cite[Definition 3]{am:mefrwi}) that a module $P\in R{-\rm Mod}$ is said to be \textit{locally projective} if there is a direct system $(P_i)_{i\in I}$ of finitely generated projective direct summands of $P$ (so that $I$ is a directed set, and $P_i$ is a direct summand of $P_j$ whenever $i\leq j$) 
such that $\varinjlim_I P_i = P$.

Write  $\psi_i: P\longrightarrow P_i$  for the projection.  For a locally projective module $P\in R{-\rm Mod}$, the endomorphisms of each $P_{i}$ extend to endomorphisms of $P$ when composed by $\psi_i$, and in this way the endomorphism rings of the components $P_{i}$ form a direct system of subrings of
${\rm End}_{R}(P)$. Their limit $\varinjlim_I {\rm End}_{R}(P_{i})$ consists exactly of those endomorphisms of $P$ which factor through one of the projections $\psi_{i}$. The ring 
$\varinjlim_I {\rm End}_{R}(P_{i})$  has local units because if the endomorphism $\varphi \in \varinjlim_I {\rm End}_{R}(P_{i})$ factors through $\psi_{i}$, then 
 the projection $\psi_{i}$ is a unit for $\varphi$. 

For  rings with local units $R$ and $S$, we call $R$ and $S$ Morita equivalent in case the categories $R$-Mod and $S$-Mod are equivalent. In \cite[Theorem 2.5]{am:mefrwi} \'{A}nh and M\'{a}rki proved that two rings $R$ and $S$ with local units are Morita equivalent if and only if there exists a locally projective generator $P$ for $R{-\rm Mod}$ such that, using the notation above, $S\cong \varinjlim_I {\rm End}_{R}(P_{i})$. Applying this result, Corollary \ref{limofcor} and Theorem \ref{End(Q)general}, we close this section with the following result and subsequent example.

\begin{thm}\label{moquiLeavitt}
Let $K$ be any field and let $A$ be a $K$-algebra with local units which is Morita equivalent to the Leavitt path algebra of a row-countable graph. Then $A$ is isomorphic to a Steinberg algebra.
\end{thm}
\begin{proof}
Let $A$ be a $K$-algebra with local units which is Morita equivalent to the Leavitt path algebra $L_K(F)$ of a row-countable graph $F$. Let $E$ be the Drinen-Tomforde desingularization of $F$ (see, e.g. \cite{dm:tcaoag} or \cite[p. 434]{ap:tlpaoag08}). Since the Drinen-Tomforde desingularization of a row-countable graph is always row-finite, $E$ is a row-finite graph. By \cite[Theorem 5.2]{ap:tlpaoag08} (or \cite[Theorem 14]{ar:reeofrg}), the Leavitt path algebras $L_K(E)$ and $L_K(F)$ are Morita equivalent, whence $A$ is Morita equivalent to $L_K(E)$. Then, by \cite[Theorem 2.5]{am:mefrwi}, there exists a locally projective generator $P$ for $L_K(E){-\rm Mod}$ such that, using the notation above, $A\cong \varinjlim_I {\rm End}_{L_K(E)}(P_{i})$ as $K$-algebras. 

For each $i\in I$, by Lemma \ref{projmod}, we have that $P_i \cong \bigoplus_{v\in V_i} n^i_vL_K(E)v$ for some finite subset $V_i$ of $E^0$ and some positive integers $(n^i_v)_{v\in V_i}$. We note that for all $j\geq i$ in $I$, we have that $P_i$ is a direct summand of $P_j$, whence we obtain a direct sum decomposition of $P_j$ as follows: $P_j \cong \bigoplus_{v\in V_j} n^j_vL_K(E)v$ for some finite subset $V_j$ of $E^0$ and some positive integers $(n^j_v)_{v\in V_j}$ satisfying $V_i\subseteq V_j$ and $n^i_v\leq n^j_v$ for all $v\in V_i$.

Let $G_i$ be the graph formed from $E$ by taking each $v\in V_i$ and attaching a ``head" of length $n^i_v-1$ of the form
$$\xymatrix{  \bullet^{v^{n^i_v-1}}\ar[r]^{e_v^{n^i_v-1}} & \cdots \bullet^{v^2} \ar[r]^{e_v^2} & \bullet^{v^1} \ar[r]^{e_v^1} &\bullet^{v}}$$ to $E$. Put $$T_i := V_i \cup \{v^j\mid v\in V_i, 1\leq j\leq n^i_v-1\}\subseteq G_i^0.$$ As was shown in the proof of Theorem \ref{End(Q)general}, we have that $${\rm End}_{L_K(E)}(P_{i}) \cong (\sum_{v\in T_i}v)L_K(G_i)(\sum_{v\in T_i}v)$$ as $K$-algebras.

By the  above note and the construction of graphs $G_i$ $(i\in I)$, we have that $G_i$ is a CK-subgraph of $G_j$ and $T_i\subseteq T_j$ whenever $i\leq j$; that means, we have that $(G_i)_{i\in I}$ is a direct system in \textbf{CKGr} and the set $\{T_i\subseteq G^0_i\mid i\in I\}$ satisfies the condition as in Corollary \ref{limofcor}. Let $G$ be the direct limit for the system $(G_i)_{i\in I}$ in \textbf{CKGr} with canonical morphisms $\eta_i: G_i\longrightarrow G$, and $T= \bigcup_{i\in I}\eta^0_i(T_i)\subseteq G^0$. By Corollary \ref{limofcor}, we get that 
\begin{center}
$A\cong \varinjlim_{I}{\rm End}_{L_K(E)}(P_{i})\cong \varinjlim_{I} (\sum_{v\in T_i}v)L_K(G_i)(\sum_{v\in T_i}v)\cong A_K(\mathcal{G}_G|_T)$ 
\end{center} as $K$-algebras, thus showing that $A$ is isomorphic to a Steinberg algebra,  finishing the proof.
\end{proof}	

\begin{exas}
We note that the row-countable graph $C$ of Example \ref{cornernotleav} provides an example of a $K$-algebra which is Morita equivalent to the Leavitt path algebra of a row-countable graph, but is not isomorphic to a Leavitt path algebra.   Specifically, because $w_n = e_n^* v e_n$ for each $n\in \mathbb{N}$, we see that each vertex of $C$ lies in the ideal $L_K(C) v L_K(C)$ of $L_K(C)$, whence $v$ is a full idempotent in $L_K(C)$.  Thus  \cite[Theorem 2.5]{am:mefrwi} applies, and we conclude  that $v L_K(C)v$ is Morita equivalent to $L_K(C)$.   But, as shown in Example \ref{cornernotleav}, the algebra $v L_K(C)v$ is not isomorphic to a Leavitt path algebra.  
\end{exas}

\section{Corners of graph $C^*$-algebras}
Arklint and Ruiz \cite{ar:cocka},  and Arklint, Gabe and Ruiz  \cite{ArklintGabeRuiz}  have established (among many other things) that for a countable graph $E$ having finitely many vertices, any corner $pC^*(E)p$ of the graph $C^*$-algebra $C^*(E)$ by a projection $p$ is isomorphic to a graph $C^*$-algebra. The goal of this section is to show that for a countable graph $E$, any corner $pC^*(E)p$ of the graph $C^*$-algebra $C^*(E)$ by a projection $p$ is isomorphic to a $C^*$-algebra of an ample groupoid (Theorem \ref{CorC-algtheo}).

Following \cite[Theorem 1.2]{kpr:ckaodg}, if $E$ is a countable graph, then the \textit{graph $C^*$-algebra} $C^*(E)$ is the universal $C^*$-algebra generated by mutually orthogonal projections $\{p_v\mid v\in E^0\}$ and partial isometries with mutually orthogonal ranges $\{s_e\mid e\in E^1\}$ satisfying
\begin{itemize}
	\item[(1)] $s^*_e s_e = p_{r(e)}$ for all $e\in E^1$,
	\item[(2)] $p_v= \sum_{e\in s^{-1}(v)}s_es^*_e$ for any regular vertex $v$,
	\item[(3)] $s_es^*_e\leq p_{s(e)}$ for all $e\in E^1$.
\end{itemize}
If $\mu = e_1\cdots e_n\in E^*$ and $n\geq 2$, then we let $s_{\mu}:= s_{e_1}\cdots s_{e_n}$. Likewise, we let $s_v : = p_v$ if $v\in E^0$. It follows from \cite[Lemma 1.1]{kpr:ckaodg} that $$C^*(E) = \overline{\text{Span}} \{s_{\mu}s^*_{\nu} \mid \mu, \nu \in E^*,\, r(\mu) = r(\nu)\}.$$ 
Take the projection $p = \sum_{v\in H}p_v$ where $H$ is a finite subset of $E^0$. It is crucial for us to observe that 
\begin{equation*}
p s_{\mu}s^*_{\nu}=  \left\{
\begin{array}{lcl}
s_{\mu}s^*_{\nu}&  & \text{if } s(\mu)\in H  \\
0&  & \text{otherwise}%
\end{array}%
\right.
\end{equation*}%
so that 
$$pC^*(E)p = \overline{\text{Span}} \{s_{\mu}s^*_{\nu} \mid \mu, \nu \in E^*,\, s(\mu), s(\nu)\in H, \, r(\mu) = r(\nu)\}.$$

We recall for reader's convenience that a map between  $C^*$-algebras is a $C^*$-algebra isomorphism if and only if it is an isomorphism of $*$-algebras (see, for example, \cite[Theorem 2.1.7]{Murphy}). Accordingly, throughout this section the symbol $\cong$ will mean ``isomorphism as $C^*$-algebras".

In \cite{kprr:ggacka}, Kumjian, Pask, Raeburn, and Renault defined the groupoid $C^*$-algebra $C^*(\mathcal{G}_E)$ associated to a row-finite graph with no sources $E$. In \cite{bcw:gaaoe}, Brownlowe, Carlsen and Whittaker explained this construction in the case when $E$ is a countable graph. Also, in \cite[Proposition 2.2]{bcw:gaaoe}, Brownlowe, Carlsen and Whittaker showed that if $E$ is a countable graph, then there is a unique isomorphism $\pi: C^*(E)\longrightarrow C^*(\mathcal{G}_E)$ such that $\pi(p_v) = 1_{Z(v, v)}$ for all $v\in E^0$ and $\pi(s_e) = 1_{Z(e, r(e))}$ for all $e\in E^1$.

\begin{prop}\label{corC-alg} Let $E$ be a countable graph, $H$ a nonempty finite subset of $E^0$, and $p = \sum_{v\in H}p_v\in C^*(E)$. Then $pC^*(E)p \cong C^*(\mathcal{G}_E|_H)$.	
\end{prop}
\begin{proof} We note firstly that $$pC^*(E)p = \overline{\text{Span}} \{s_{\mu}s^*_{\nu} \mid \mu, \nu \in E^*,\, s(\mu), s(\nu)\in H, \, r(\mu) = r(\nu)\}.$$
By Theorem \ref{gecorthem}, the sets $Z(\mu, \nu, F)$, where $\mu, \nu\in E^*$ with $s(\mu), s(\nu)\in H$, $r(\mu) = r(\nu)$, and a finite subset $F \subseteq s^{-1}(r(\mu))$, constitute a base of compact open bisections for the topology of $\mathcal{G}_E|_H$, and by \cite[Proposition 4.2]{cfst:aggolpa} the span of the characteristic functions of those sets is dense in $C^*(\mathcal{G}_E|_H)$. From this and by the above isomorphism $\pi: C^*(E)\longrightarrow C^*(\mathcal{G}_E)$, we immediately obtain that $pC^*(E)p \cong C^*(\mathcal{G}_E|_H)$, thus finishing the proof.	
\end{proof}

For a $C^*$-algebra $A$, let $M_{\infty}(A)$ be the directed union of $M_n(A)$ $(n\in \mathbb{N})$, where the transition maps $M_n(A)\longrightarrow M_{n+1}(A)$ are given by $x\longmapsto \left(
\begin{array}{cc}
x & 0 \\
0 & 0 \\
\end{array}
\right).$ We define $V(A)$ to be the set of Murray-von Neumann equivalence classes (denoted $[P]$) of projections in $M_{\infty}(A)$; see \cite[4.6.2 and 4.6.4]{b:ktfoa}. If $P$ and $Q$ are projections in
$M_{\infty}(A)$, we will use the symbol $P \sim Q$ to indicate that they are (Murray-von Neumann) equivalent, that is, there is a partial isometry $W \in M_{\infty}(A)$ such that $W^*W = P$ and $WW^* = Q$. We will write $P\oplus Q$ for the block-diagonal matrix  $\text{diag}(P, Q)$, and we will denote by $n\cdot P$ the direct sum of $n$ copies of $P$. We endow $V(A)$ with the structure of a commutative monoid by imposing the operation \[[P] + [Q] = [P\oplus Q]\] for all projections $P$ and $Q\in M_{\infty}(A)$.

In \cite[Theorem 7.3]{tomf:utaisflpa}, Tomforde showed that if $E$ is a countable graph, then there exists an injective algebra $*$-homomorphism $\iota_E: L_{\mathbb{C}}(E)\longrightarrow C^*(E)$
with $\iota_E(v) = p_v$ and $\iota_E(e) = s_e$ for all $v\in E^0$ and $e\in E^1$. Hay et. al \cite{hlmrt:nsktflpa} proved the following interesting result.

\begin{thm}[{\cite[Corollary 3.5]{hlmrt:nsktflpa}}]\label{HLMRTthm}
	For any countable graph $E$, the natural injection 	$\iota_E: L_{\mathbb{C}}(E)\longrightarrow C^*(E)$ induces a monoid
	isomorphism $V(\iota_E): V(L_{\mathbb{C}}(E))\longrightarrow V(C^*(E))$. 
	
\end{thm}

Using Theorem \ref{HLMRTthm} and Lemma~\ref{projmod}, we obtain the following useful corollary.

\begin{cor}\label{projections}
	For any countable graph $E$, every nonzero projection $p\in C^*(E)$ may be written in the form	$$p\sim (\bigoplus_{v\in V} n_v (p_v - \sum_{e\in T_v}s_es^*_e)) \oplus (\bigoplus_{w\in W} n_w p_w) \quad\text{in}\quad M_{\infty}(C^*(E)),$$ where $V$ and $W$ are finite subsets of $E^0$, each $v\in V$ is an infinite emitter, each $T_v$ is a nonempty finite subset of $s^{-1}(v)$, and the numbers $n_v, n_w$ are positive integers.
\end{cor}

We proceed with the next fact.

\begin{lem}\label{outsplitlem}
	Let $E$ be a countable graph, $v$ an infinite emitter and $T_v$ a nonempty finite subset of $s^{-1}(v)$. Put $\mathcal{E}_1 = T_v$ and $\mathcal{E}_2 = s^{-1}(v)\setminus T_v$. Let $F$ be the graph obtained by out-splitting the vertex $v$ into the vertices $v^1, v^2$ according to  the partition $\mathcal{E}_1, \mathcal{E}_2$. Then there exists a $C^*$-algebra isomorphism
	$\Phi_v: C^*(E)\longrightarrow C^*(F)$ with the following properties:
	\begin{itemize}	
		\item[(1)] $\Phi_v(p_v) = p_{v^1}+ p_{v^2}$ and $\Phi_v(p_w) = p_w$ for all $w\in E^0\setminus \{v\}$,
		\item[(2)] For all $w\in E^0\setminus \{v\}$ and all finite subsets $W\subseteq s^{-1}(w)$ with $v\notin r_E(W)$,
		$$\Phi_v(p_w - \sum_{e\in W}s_es_e^*) = p_w - \sum_{e\in W}s_es_e^*,$$
		\item[(3)] For all $w\in E^0\setminus \{v\}$ and any finite subset $W\subseteq s^{-1}(w)$ with $v\in r_E(W)$, there
		exists a finite subset $W'\subseteq s_F^{-1}(w)$ such that
		$$\Phi_v(p_w - \sum_{e\in W}s_es_e^*) = p_w - \sum_{e\in W'}s_es_e^*,$$
		\item[(4)] $\Phi_v(p_v - \sum_{e\in T_v}s_es_e^*) = p_{v^2}$.
	\end{itemize}
\end{lem}
\begin{proof}
	By \cite[Theorem 3.2]{bp:feoga}, the map $\Phi_v: C^*(E)\longrightarrow C^*(F)$, defined by 
	\begin{equation*}
	\Phi_v(p_u)=  \left\{
	\begin{array}{lcl}
	p_{v^1} + p_{v^2}&  & \text{if } u = v , \\
	p_u&  & \text{otherwise \ \ ,}%
	\end{array}%
	\right.
	\end{equation*}%
	
	and 	
	\begin{equation*}
	\Phi_v(s_e)=  \left\{
	\begin{array}{lcl}
	s_{e^1} + s_{e^2}&  & \text{if } e\in r^{-1}(v)\\
	s_e&  & \text{otherwise}  \ \ , \ \ \ \ \   
	\end{array}%
	\right.
	\end{equation*}%
	extends to a  $C^*$-algebra isomorphism. Then, similarly as in the proof of Corollary \ref{outsplitcor}, we get immediately that $\Phi_v$ satisfies the properties as in the statement, finishing the proof.
\end{proof}

Using Corollary \ref{projections} and  Lemma \ref{outsplitlem}  we obtain the following result which plays an important role in the proof of the main theorem below.

\begin{prop}\label{repofproj}
	Let $E$ be a countable graph and $p$ a nonzero projection in $C^*(E)$. 	Then there exists a graph $F$ with the following properties:
	\begin{itemize}	
		\item[(1)] $F$ is obtained from $E$ in some step-by-step process of out-splittings.  
		\item[(2)] There exists a $C^*$-algebra isomorphism $\Phi: C^*(E) \longrightarrow C^*(F)$ such that $\Phi(p) \sim \bigoplus_{v\in H}n_v p_v$ in $M_{\infty}(C^*(F))$ for some finite subset $H\subseteq F^0$ and some positive integers $\{n_v\}_{v\in H}$.
	\end{itemize}
\end{prop}
\begin{proof}
	By Corollary \ref{projections}, we have that
	$$p\sim (\bigoplus_{v\in V} n_v (p_v - \sum_{e\in T_v}s_es^*_e)) \oplus (\bigoplus_{w\in W} n_w p_w) \quad\text{in}\quad M_{\infty}(C^*(E)),$$ where $V$ and $W$ are some finite subsets of $E^0$, each $v\in V$ is an infinite emitter, each $T_v$ is some nonempty finite subset of $s^{-1}(v)$, and  $n_v, m_w$ are some positive integers.	Write $V = \{v_1, v_2, \hdots, v_n\}$. Let $E_1$ be the graph obtained from $E$ by out-splitting the vertex $v_1$ into the vertices $v_1^1, v_1^2$ according to the partition $\mathcal{E}_1 = T_{v_1}$, $\mathcal{E}_2 = s^{-1}(v_1)\setminus T_{v_1}$. By Lemma \ref{outsplitlem}, there exists a $C^*$-algebra isomorphism
	$\Phi_{v_1}: C^*(E)\longrightarrow C^*(E_1)$ with the following properties:
	\begin{itemize}	
		\item[(i)] $\Phi_{v_1}(p_{v_1}) = p_{v^1_1}+ p_{v^2_1}$ and $\Phi_v(p_w) = p_w$ for all $w\in E^0\setminus \{v_1\}$,
		\item[(ii)] for all $w\in E^0\setminus \{v_1\}$ and all finite subsets $W\subseteq s^{-1}(w)$ with $v_1\notin r_E(W)$,
		$$\Phi_{v_1}(p_w - \sum_{e\in W}s_es_e^*) = p_w - \sum_{e\in W}s_es_e^*,$$
		\item[(iii)] for all $w\in E^0\setminus \{v_1\}$ and all finite subsets $W\subseteq s^{-1}(w)$ with $v\in r_E(W)$, there
		exists a finite subset $W'\subseteq s_{E_1}^{-1}(w)$ such that
		$$\Phi_{v_1}(p_w - \sum_{e\in W}s_es_e^*) = p_w - \sum_{e\in W'}s_es_e^*, \ \mbox{and}$$
		\item[(iv)] $\Phi_{v_1}(p_{v_1} - \sum_{e\in T_{v_1}}s_es_e^*) = p_{v_1^2}$.
	\end{itemize}
	We then have that 
	$$\Phi_{v_1}(p) \sim (\bigoplus^n_{i=2}n_{v_i}(p_{v_i} - \sum_{e\in T'_{v_i}}s_es_e^*))\oplus n_{v_1} p_{v_1^2} \oplus (\bigoplus_{w\in W''} m''_w p_w)$$ in $M_{\infty}(C^*(E_1))$, where $T'_{v_i}$ is some nonempty finite subset of $s^{-1}_{E_1}(v_i)$, and $W'', m''_w$ are defined by setting $$W'' = W \mbox{ if } v_1\notin W $$
	\noindent
	(in this case we have that $ m''_w = m_w$  for all $ w\in W$), and 
	$$W'' = \{v^1_1, v^2_1\} \cup W\setminus\{v_1\} \mbox{ otherwise }$$
	\noindent
	(in this case we have that $ m''_w = m_w  $ for all $ w\in W\setminus\{v_1\} $ and $m''_{v_1^i} = m_{v_1}).$
	
	Let $W' := W'' \cup \{v_1^2\}$, $m'_w := m''_w$ for all $w\in W'\setminus\{v^2_1\}$, and $m'_{v^2_1} := n_{v_1} + m''_{v_1^i}$. We then get that
	$$\Phi_{v_1}(p)\sim (\bigoplus^n_{i=2}n_{v_i} (p_{v_i} - \sum_{e\in T'_{v_i}}s_es_e^*))\oplus (\bigoplus_{w\in W'} m'_w p_w)$$ in $M_{\infty}(C^*(E_1))$.
	
	We repeat the process described above, starting with the graph $E_2$ which is obtained from $E_1$ by out-splitting $v_2$ into $v_2^1$ and $v_2^2$ with respect to  $\mathcal{E}_1 = T'_{v_2}$, $\mathcal{E}_2 = s^{-1}_{E_1}(v_2)\setminus T'_{v_2}$. We see that after $n$ steps we arrive at the  graph $F$ of the statement and a $C^*$-algebra isomorphism
	$\Phi: C^*(E)\longrightarrow C^*(F)$ such that $\Phi(p) \sim \bigoplus_{v\in H}n_v p_v$ in $M_{\infty}(C^*(F))$ for some finite subset $H\subseteq F^0$ and some positive integers $\{n_v\}_{v\in H}$, finishing the proof.	
\end{proof}	

We are now in position to achieve the  main result of this section. Before doing so, we need to recall the concept of the stabilization of a graph; see \cite[Definition 9.4]{at:iameoga}. Given a graph $E$, let $SE$ be the graph formed from $E$ by taking
each $v\in E^0$ and attaching an infinite ``head" of the form

$$\xymatrix{\cdots  \ar[r]^{e_v^{4}} & \bullet^{v^{3}}\ar[r]^{e_v^{3}} &  \bullet^{v^2} \ar[r]^{e_v^2} & \bullet^{v^1} \ar[r]^{e_v^1} &\bullet^{v}}$$ to $E$. We call $SE$ the \textit{stabilization} of $E$.

\begin{thm}\label{CorC-algtheo}
Let $E$ be a countable graph and $p$ a nonzero projection in $C^*(E)$. There exists a countable graph $F$ with the following properties:
\begin{itemize}	
		\item[(1)] $F$ is obtained from $E$ in some step-by-step process of out-splittings.  
		\item[(2)] $pC^*(E)p \cong C^*(\mathcal{G}_{SF}|_H)$ for some finite subset $H\subseteq (SF)^0$.
	\end{itemize} 	
\end{thm}
\begin{proof}
	By Proposition \ref{repofproj}, there exists a graph $F$ which 	 is obtained from $E$ in some step-by-step process of out-splittings, and a $C^*$-algebra isomorphism $\Phi: C^*(E)\longrightarrow C^*(F)$ such that $q:=\Phi(p) \sim \bigoplus_{v\in T}n_v p_v$ in $M_{\infty}(C^*(F))$ for some finite subset $T\subseteq F^0$ and some positive integers $\{n_v\}_{v\in T}$.
	
	Let $\mathcal{K}$ be  the compact operators on a separable infinite-dimensional Hilbert space $H$. By \cite[Proposition 9.8]{at:iameoga} and its proof, we have that $C^*(F)\otimes \mathcal{K} = \overline{M_{\infty}(C^*(F))}$, and there exists a $C^*$-algebra isomorphism $\phi: C^*(F)\otimes \mathcal{K}\longrightarrow C^*(SF)$ such that $\phi(p_v\otimes E_{11}) = p_v$ and $\phi(p_v\otimes E_{(k+1)(k+1)}) = p_{v^k}$ for all $v\in F^0$ and $k\geq 1$. Write $T = \{v_1, v_2, \dots , v_u\}$ and let $n_i := n_{v_i}$ for all $v_i\in T$.
	Let $H := T \cup \{v_i^k\mid 1\leq i\leq u, 1\leq k\leq n_i-1\}\subseteq (SF)^0$. We then have $\phi(q)$ is Murray-von Neumman equivalent to $\sum_{v\in H}p_v$ in $C^*(SF)$, and hence $\phi(q)C^*(SF)\phi(q) \cong (\sum_{v\in H}p_v) C^*(SF)(\sum_{v\in H}p_v)$. By Proposition~\ref{corC-alg}, $(\sum_{v\in H}p_v) C^*(SF)(\sum_{v\in H}p_v) \cong C^*(\mathcal{G}_{SF}|_{H})$, so
	$\phi(q)C^*(SF)\phi(q) \cong C^*(\mathcal{G}_{SF}|_{H})$.
	
	On the other hand, we note that $$pC^*(E)p \cong qC^*(F)q \cong q(C^*(F)\otimes \mathcal{K})q\cong \phi(q)C^*(SF)\phi(q),$$ so $pC^*(E)p\cong C^*(\mathcal{G}_{SF}|_{H}),$ thus finishing the proof.
\end{proof}

\vskip 0.5 cm \vskip 0.5cm {

\end{document}